\documentclass{amsart}
\usepackage{amsthm,amsmath,amssymb,amscd,graphics,enumerate, stmaryrd,xspace,verbatim, epic, eepic,url}

\usepackage[usenames,dvipsnames]{color}

\usepackage[colorlinks=true]{hyperref}

\usepackage[active]{srcltx}
\usepackage[all]{xypic}
\SelectTips{cm}{}

{
\newtheorem{theoremintro}{Theorem}
   \newtheorem{theorem}[subsubsection]{Theorem}
      \newtheorem*{theorem*}{Theorem}

   \newtheorem{lemma}[subsubsection]{Lemma}

   \newtheorem{corollary}[subsubsection]{Corollary}
   
   \newtheorem*{conjecture*}{Conjecture}

}
{\theoremstyle{definition}
          \newtheorem*{exercise*}{Exercise}
   
   \newtheorem{example}[subsubsection]{Example}
   \newtheorem*{example*}{Example}

   \newtheorem*{definition*}{Definition}
   
   \newtheorem{remark}[subsubsection]{Remark}

}
\newcommand{\NN}{{\mathbb{N}}}

\newcommand{\ZZ}{{\mathbb{Z}}}

\newcommand{\GG}{{\mathbb{G}}}

\def\alp{\alpha}

\newcommand{\cC}{{\mathcal C}}

\newcommand{\tilcC}{{\widetilde{\mathcal C}}}
\newcommand{\tilcF}{{\widetilde{\mathcal F}}}
\newcommand{\tilcT}{{\widetilde{\mathcal T}}}
\newcommand{\cF}{{\mathcal F}}

\newcommand{\cM}{{\mathcal M}}

\newcommand{\cO}{{\mathcal O}}

\newcommand{\cT}{{\mathcal T}}

\newcommand{\cX}{{\mathcal X}}

\newcommand{\oM}{{\overline{M}}}

\newcommand{\ox}{{\overline{x}}}

\newcommand{\oz}{{\overline{z}}}

\newcommand{\veps}{{\varepsilon}}
\newcommand{\oveps}{{\overline{\veps}}}

\newcommand{\Fan}{{\rm Fan}}
\newcommand{\tX}{{\widetilde{X}}}

\def\<{\langle}
\def\>{\rangle}

\newcommand{\Spec}{\operatorname{Spec}}

\newcommand{\cs}{{\operatorname{cs}}}
\newcommand{\tcs}{{\operatorname{tcs}}}

\newcommand{\Ker}{{\operatorname{Ker}}}

\newcommand{\bfA}{{\mathbf A}}

\newcommand{\oX}{{\overline{X}}}
\newcommand{\oY}{{\overline{Y}}}
\newcommand{\oZ}{{\overline{Z}}}
\newcommand{\oU}{{\overline{U}}}
\newcommand{\oR}{{\overline{R}}}
\newcommand{\oW}{{\overline{W}}}

\def\longto{\dashrightarrow}
\def\toto{\rightrightarrows}
\def\:{{\colon}}
\def\.{{,\dots,}}

\newcommand{\double}{\genfrac..{0pt}1
{\raise -1pt\hbox{$\scriptstyle\longrightarrow$}}{\raise 3pt\hbox
{$\scriptstyle\longrightarrow$}}}

\renewcommand{\setminus}{\smallsetminus}

\def\sat{{\rm sat}}
\def\nor{{\rm nor}}

\def\fl{_{\rm fl}}
\def\et{_{\rm \acute et}}
\def\ket{_{\rm k\acute et}}

\def\tototi{\mathbin{\mathop{\otimes}\limits^{\raise-1pt\hbox
{$\scriptscriptstyle {\rm L}$}}}}

\def\indlim{\mathop{\vrule width0pt height7pt depth
4pt\smash{\lim\limits_{\raise 1pt\hbox to 14.5pt
{\rightarrowfill}}}}}
\def\projlim{\mathop{\vrule width0pt height7pt depth
4pt\smash{\lim\limits_{\raise 1pt\hbox to 14.5pt
{\leftarrowfill}}}}}

\newcommand\displaceamount{3pt}

\newcommand{\doubledown}{\ar@<\displaceamount>[d]\ar@<-\displaceamount>[d]}

\newcommand{\doubleup}{\ar@<\displaceamount>[u]\ar@<-\displaceamount>[u]}

\newcommand{\doubleright}{\ar@<\displaceamount>[r]\ar@<-\displaceamount>[r]}

\newcommand{\tor}{{\operatorname{tor}}}

\newcommand{\rk}{\operatorname{rk}}

\def\into{\hookrightarrow}
\def\onto{\twoheadrightarrow}

\def\gp{\text{gp}}

\def\tilX{{\widetilde X}}
\def\tilJ{{\widetilde J}}
\def\toisom{\xrightarrow{{_\sim}}}

\begin{document}
\title{Toroidal orbifolds, destackification, and Kummer blowings up}

\author[D. Abramovich]{Dan Abramovich}
\address{Department of Mathematics, Box 1917, Brown University,
Providence, RI, 02912, U.S.A}
\email{abrmovic@math.brown.edu}
\author[M. Temkin]{Michael Temkin}
\address{Einstein Institute of Mathematics\\
               The Hebrew University of Jerusalem\\
                Giv'at Ram, Jerusalem, 91904, Israel}
\email{temkin@math.huji.ac.il}

\author[J. W{\l}odarczyk] {Jaros{\l}aw W{\l}odarczyk}
\address{Department of Mathematics, Purdue University\\
150 N. University Street,\\ West Lafayette, IN 47907-2067}
\thanks{This research is supported by  BSF grant 2014365}

\date{\today}

\begin{abstract}
We show that any toroidal DM stack $X$ with finite diagonalizable inertia possesses a maximal toroidal coarsening $X_\tcs$ such that the morphism $X\to X_\tcs$ is logarithmically smooth. 

Further, we use torification results of \cite{AT1} to construct a destackification functor, a variant of the main result of \cite{Bergh}, on the category of such toroidal stacks $X$. Namely, we associate to $X$ a sequence of blowings up of toroidal stacks $\tilcF_X\:Y\longrightarrow X$ such that $Y_\tcs$ coincides with the usual coarse moduli space $Y_\cs$. In particular, this provides a toroidal resolution of the algebraic space $X_\cs$. 

Both $X_\tcs$ and $\tilcF_X$ are functorial with respect to strict inertia preserving morphisms $X'\to X$. 

Finally, we use  coarsening morphisms to introduce a class of non-representable birational modifications of toroidal stacks called Kummer blowings up. 

These modifications, as well as our version of destackification, are used in our  work on functorial toroidal resolution of singularities.
\end{abstract}
\maketitle
\setcounter{tocdepth}{1}

\tableofcontents

\section{Introduction}
We study the birational geometry of toroidal orbifolds, aiming towards applications in resolution of singularities and semistable reduction, as initiated in our  paper \cite{ATW-principalization}.

Starting from a toroidal Deligne--Mumford stack $X$ with diagonalizable inertia, we prove the following destackification result:

\begin{theoremintro}[{See Theorem \ref{destacksimpleth}}]
Let $\cC$ be the category of toroidal DM stacks with finite diagonalizable inertia acting trivially on the sharpened stalks $\oM_x$ of the logarithmic structure. Then to any object $X$ in $\cC$ one can associate a {\em destackifying blowing up} of toroidal stacks $\cF_X\:X'\to X$ along an ideal $I_X$ and a {\em coarse destackifying blowing up} $\cF^0_X\:X_0\to X_\cs$ along an ideal $J_X$ so that

(i) $X_0=(X')_\cs$ and $X_0$ inherits from $X'$ a logarithmic structure making it a toroidal scheme such that the morphism $X' \to X_0$ is logarithmically smooth.

(ii) The blowings up are compatible with any surjective inert morphism $f\:Y\to X$ from $\cC$, which is either strict or logarithmically smooth: $I_X\cO_Y=I_Y$, $J_X\cO_{Y_\cs}=J_Y$, $Y'=X'\times_XY$ and $Y'_0=X'_0\times_{X_\cs}Y_\cs$.
\end{theoremintro}

In addition, we remove the assumption on the triviality of the inertia action in Theorem~\ref{destackth}. In this case, destackification is achieved by a sequence of blowings up, which is only compatible with strict inert morphisms.

The theorem above is a variant of the main result of \cite{Bergh}. It is tuned for different purposes and uses different methods. First, we restrict to diagonalizable inertia, Theorem \ref{destacksimpleth} generalized the main result of \cite{Bergh} in two directions: we allow arbitrary toroidal singularities, and we do not restrict to stacks of finite type over a field. Our method is also different from Bergh's, in that we use the \emph{torific ideal} of \cite{AT1} which produces the destackification result in one step. Unlike Bergh's result we do not describe the destackification in terms of a sequence of well-controlled operations such as blowings up and root stacks, in particular, applications to factorization of birational maps must use \cite{Bergh} rather than our theorems.

Our study of destackification requires understanding the degree to which one may remove stack structure while keeping logarithmic smoothness. For this purpose we introduce and study coarsening morphisms of Deligne--Mumford stacks in general in Section \ref{Sec:coarsening}, and then specialize in Section \ref{Sec:toroidal} to toroidal stacks, where we associate to a toroidal Deligne--Mumford stack $X$ its \emph{total toroidal coarsening} $X_\tcs$ and prove

\begin{theoremintro}[{See Theorem \ref{modulith}}]
Let $\cC$ be the 2-category of toroidal DM stack with finite diagonalizable inertia and let $X$ be an object of $\cC$. Then,

(i) The total toroidal coarsening $X \to X_\tcs$ exists.

(ii) For any geometric point $x\to X$, we have $(I_{X/X_\tcs})_x = G_x^\tor$, where $(I_{X/X_\tcs})_x$ is the relative stabilizer and $G_x^\tor\subset G_x$ the maximal subgroup of inertia acting toroidally.

(iii) Any logarithmically flat morphism $h\:Y\to X$ in $\cC$ induces a morphism $h_\tcs\:Y_\tcs\to X_\tcs$ with a 2-commutative diagram
$$
\xymatrix{
Y\ar[d]_h\ar[r]^{\phi_Y} & Y_\tcs\ar[d]^{h_\tcs} \\
X\ar[r]_{\phi_X}\ar@{=>}[ur]^\alpha & X_\tcs
}
$$
and the pair $(h_\tcs,\alpha)$ is unique in the 2-categorical sense.

(iv) Assume that $h$ is logarithmically flat and inert. Then the diagram in (iii) is 2-cartesian.
\end{theoremintro}

Apart from destackification, this theorem figures in our study of a collection of non-representable birational modifications which is essential in our work \cite{ATW-principalization} on resolution of singularities. We define in Section \ref{Sec:permissible-center} the notion of a \emph{permissible Kummer center $I$} on a toroidal scheme, and in Section \ref{Sec:Kummer-blowup} we define its blowing up $[Bl_I(X)]\to X$, which is in general  an algebraic stack. Its key property is the following:

\begin{theoremintro}[{See Theorem \ref{Th:permissible-Kummer}}]
Let $X$ be a toroidal scheme and let $I$ be a permissible Kummer ideal with the associated Kummer blowing up $f\:[Bl_I(X)]\to X$. Then

(i) (Principalization property) $f^{-1}(I)$ is an invertible Kummer ideal.

(ii) (Exceptional behavior) If $I$ is not monomial at $x\in X$ then $f^{-1}(I)$ is an invertible \emph{ideal} along $f^{-1}(x)$. If $I$ is monomial at $x$, say $I=(m_1^{1/d}\.m_r^{1/d})$ then $[Bl_I(X)]=Bl_J(X)$, where $J=(m_1\.m_r)$.

(iii) (Universal property) If $I$ is monomial then $f$ is the universal morphism of toroidal orbifolds $h\:Z\to X$ such that $h^{-1}(I)$ is an invertible \emph{Kummer} ideal.  If $I$ is not monomial at any $x\in X$ then  $f$ is the universal morphism of toroidal orbifolds $h\:Z\to X$ such that $h^{-1}(I)$ is an invertible \emph{ideal}.
\end{theoremintro}

\section{Coarsening morphisms and inertia}

\subsection{Inertia stack}

\subsubsection{Basic properties of inertia}\label{Sec:basic-inertia}
Recall that the inertia stack $I_{X/Y}$ of a morphism $f\:X\to Y$ of stacks is the second diagonal stack $I_{X/Y}=X\times_{\Delta_{X/Y}}X$, where $\Delta_{X/Y}=X\times_YX$. It is a representable group object over $X$.

The absolute inertia stack of $X$ is $I_X=I_{X/\ZZ}$. Recall that by \cite[Tag:04Z6]{stacks}
\begin{equation}\label{inertiaeq1}
I_{X/Y}=I_X\times_{I_Y}X.
\end{equation}
In other words, $I_{X/Y}=\Ker(I_X\to f^*(I_Y))$, where $f^*(I_Y)=I_Y\times_YX$.

In fact, the inertia stack is a group functor in the following sense: given a morphism $f\:X\to Y$ a natural morphism $I_f\:I_X\to I_Y$ arises, and the induced morphism $I_X\to f^*(I_Y)$ is a homomorphism. In addition, the inertia functor is defined as a 2-limit and hence it respects 2-limits, including fiber products. So, given $T=X\times_ZY$ with projections $f\:T\to X$, $g\:T\to Y$ and $h\:T\to Z$, one has that
\begin{equation}\label{inertiaeq2}
I_{X\times_ZY}=I_X\times_{I_Z}I_Y=f^*(I_X)\times_{h^*(I_Z)}g^*(I_Y).
\end{equation}
Similar facts hold for relative inertia over a fixed stack $S$.

\subsubsection{Inert morphisms}
We say that a morphism $f\:X\to Y$ is {\em inert} or {\em inertia-preserving} if it respects the inertia in the sense that $I_X=f^*(I_Y)$. In particular, $I_{X/Y}=X$ and hence $f$ is representable (see \cite[Tag:04SZ]{stacks} for the absolute case, the relative case follows easily). Inert morphisms are preserved by base changes. Finally, inert morphisms have no non-trivial automorphisms.

\subsubsection{Inert groupoids}
In general, one runs into 2-categorical issues when trying to define groupoids in stacks or their quotients. The situation drastically improves if one considers inert morphisms. By a {\em an inert groupoid} in stacks we mean a usual datum $(p_{1,2}\:X_1\toto X_0,m,i,\delta)$ as in \cite[\S35.11, (Tag:0231)]{stacks}, where $X_i$ are stacks and all morphisms are inert. We crucially use that inert morphisms have no non-trivial automorphisms, and hence we can formulate the groupoid conditions using equalities of 1-morphisms without any 2-morphisms showing up. Similar consideration are pursued in \cite{Harper}.

\begin{lemma}\label{inertquot}
Assume that $p_{1,2}\:X_1\toto X_0$ is a smooth inert groupoid in Artin stacks. Then there exists an inert morphism of stacks $q\:X_0\to X$ with a 2-isomorphism $p_1\circ q=p_2\circ q$ such that $X_1=X_0\times_XX_0$. Moreover, $X$ is the quotient $[X_0/X_1]$ in the sense that any morphism $f\:X_0\to Y$ with a 2-isomorphism $f\circ p_1=f\circ p_2$ are induced by $q$ from a morphism $X\to Y$, which is unique up to a unique 2-isomorphism.
\end{lemma}
\begin{proof}
Let $U\to X_0$ be a smooth covering by a scheme and set $$R=X_1\times_{p_2,X_0} U\times_{X_0,p_1}X_1.$$ Since inert morphisms are representable, $R$ is an algebraic space and we obtain a smooth groupoid $R\toto U$ in algebraic spaces. So, a quotient $X=[U/R]$ is an Artin stack, and a (mostly 1-categorical) diagram chase shows that $X$ is as required.
\end{proof}

\subsubsection{Inertia of special types}
We say that a stack $X$ has {\em finite inertia} if the morphism $I_X\to X$ is finite, and we say that $X$ has {\em diagonalizable inertia} if the geometric fibers of $I_X\to X$ are diagonalizable groups. For example, both conditions are satisfied when $X$ admits an \'etale inert covering of the form $[Z/G]\to X$, where $Z$ is a scheme acted on by a diagonalizable group $G$.

\subsection{Coarse spaces}

\subsubsection{Coarse moduli spaces and their basic properties}
Recall that by the Keel-Mori theorem, a stack $X$ with finite inertia possesses a coarse moduli space $X_\cs$, see \cite{Keel-Mori}. A more complete stack-theoretic formulation is in \cite[Theorem~2.2.1]{Abramovich-Vistoli}. In the sequel, we will say that $X_\cs$ is the {\em coarse space} of $X$ and $X\to X_\cs$ is the {\em total coarsening morphism} of $X$. Recall that for any flat morphism of algebraic spaces $Z\to X_\cs$, the base change morphism $Y=X\times_{X_\cs}Z\to Z$ is a total coarsening morphism and the projection $Y\to X$ is flat and inert. Conversely, any inert flat morphism $h\:Y\to X$ is the base change of $h_\cs\:Y_\cs\to X_\cs$, see \cite[Lemma~2.2.2]{Abramovich-Vistoli}.

\subsubsection{The universal property}
The coarse space of $X$ is the initial morphism form $X$ to algebraic spaces, and we will extend this to stacks. We say that an inertia map $I_X\to I_Z$ is {\em trivial} if it factors through the unit $Z\to I_Z$. This happens if and only if $I_{X/Z}=I_X$.

\begin{theorem}\label{univth0}
Assume that $\phi\:X\to Z$ is a morphism of Artin stacks and the inertia of $X$ is finite. Then,

(i) The inertia map $I_\phi\:I_X\to I_Z$ is trivial if and only if $\phi$ factors through the coarse space $f\:X\to X_\cs$: there exists $\psi\:X_\cs\to Z$ and a 2-isomorphism $\alpha\:\phi\toisom\psi\circ f$.

(ii) A factorization in (i) is unique in the sense of 2-categories: if $\psi'$ and $\alp'$ form another such datum then there exists a unique 2-isomorphism $\psi=\psi'$ making the whole diagram 2-commutative.
\end{theorem}
\begin{proof}
If $\phi$ factors through $f$ then $I_\phi$ factors through the inertia $I_{X_\cs}$, which is trivial. Conversely, assume that $I_\phi$ is trivial. Choose a smooth covering of $Z$ by a scheme $Z_0$ and set $Z_1=Z_0\times_ZZ_0$ and $X_i=X\times_ZZ_i$. Since $I_{Z_i}$ and $I_\phi$ are trivial, (\ref{inertiaeq1}) and (\ref{inertiaeq2}) imply that $I_{X_i}=I_X\times_XX_i$, and we obtain that the smooth surjective morphisms $X_i\to X$ are inert. It follows that each $X_i$ has finite inertia, in particular, coarse spaces $Y_i=(X_i)_\cs$ are defined, and $Y_1\toto Y_0$ is a smooth groupoid with quotient $X_\cs$. For $i=0,1$ the map $X_i\to Z_i$ factors through $Y_i$ uniquely, and the induced morphism of groupoids $(Y_1\toto Y_0)\to(Z_1\toto Z_0)$ gives rise to the unique morphism $\psi\:X_\cs\to Z$ as required.
\end{proof}

\subsection{General coarsening morphisms}\label{Sec:coarsening}

\subsubsection{Coarsening morphisms}
We say that a morphism of stacks $X\to Y$ is a {\em coarsening morphism} if for any flat morphism $Z\to Y$ with $Z$ an algebraic space the base change $X\times_YZ\to Z$ is a total coarsening morphism. It is easy to see that coarsening morphisms are preserved by composition and arbitrary flat base change, not necessarily representable. In addition, being a coarsening morphism is a flat-local property on the target. In fact, one can show that this is the smallest class of morphisms containing total coarsening morphisms and closed under flat base changes and descent.

\begin{remark}\label{coarserem}
We use a new terminology and definition, but the object is not new. We refer to \cite[Section 3]{AOV2} for the definition of relative coarse moduli space $X_{\cs/S}$ of a morphism of stacks $X\to S$ with finite relative inertia. It is easy to see that $X\to X_{\cs/S}$ is a coarsening morphism and, conversely, for every coarsening morphism $X\to Y$ one has that $Y=X_{\cs/Y}$.
\end{remark}

\subsubsection{Basic properties}
In view of Remark~\ref{coarserem}, the following lemma is essentially covered by \cite[Theorem~3.2]{AOV2}, but we provide a proof for completeness.

\begin{lemma}\label{coarselem}
Let $X$ be an Artin stack with finite inertia and let $f\:X\to Y$ be a coarsening morphism. Then,

(i) There exists a unique morphism $g\:Y\to X_\cs$ such that $g\circ f$ is isomorphic to the total coarsening morphism $h\:X\to X_\cs$.

(ii) $f$ is a proper homeomorphism that has no non-trivial automorphisms.

(iii) $Y_\cs=X_\cs$, i.e. $g$ is the total coarsening morphism.

In particular, the 2-category of coarsening morphisms of $X$ is equivalent to a category and $h$ is its final object.
\end{lemma}
\begin{proof}
(i) Choose an atlas $Y_1\toto Y_0$ of $Y$ and set $X_i=Y_i\times_YX$. Then $Y_i=(X_i)_\cs$ and hence the composed morphisms $X_i\to X\to X_\cs$ factor uniquely through morphisms $f_i\:Y_i\to X_\cs$. The uniqueness implies that $f_1$ coincides with both pullbacks of $f_0$, hence $f$ descends to a morphism $f\:Y\to X_\cs$, which is unique.

(ii) Continuing with the notation above, since $f_i$ are total coarsening morphisms, they are proper homeomorphisms, and hence the same is true for $f$ by descent. Automorphisms of $f$ correspond to morphisms $h\:X_0\to Y_1$ such that the compositions with $Y_1\toto Y_0$ coincide. Since $Y_0=(X_0)_\cs$, any morphism $h$ factors through $Y_0$, and hence the only such $h$ is the composition of $X_0\to Y_0$ with the diagonal $Y_0\to Y_1$. This $h$ corresponds to the identity automorphism.

(iii) We should prove that a morphism $Y\to T$ with $T$ an algebraic space factors uniquely through $X_\cs$. The composed morphism $X\to Y\to T$ factors through $X_\cs$ uniquely, hence the morphisms $X_i\to X\to T$ factor through $X_\cs$. Since $Y_i=(X_i)_\cs$ we obtain that the morphisms $Y_i\to T$ factor through $X_\cs$ in a compatible way, and hence they descend to a morphism $Y\to X_\cs$ through which $Y\to T$ factors.
\end{proof}

\subsubsection{The universal property}
Similarly to coarse spaces, coarsening morphisms can be described by a universal property.

\begin{theorem}\label{univth}
Let $\phi\:X\to Z$ be a morphism of Artin stacks. Assume that $X$ is an Artin stack with finite inertia and $f\:X\to Y$ is a coarsening morphism.

(i) The following conditions are equivalent: (a) $\phi$ factors through $f$, (b) $I_\phi\:I_X\to I_Z$ factors through $I_f\:I_X\to I_Y$, (c) the map $I_{X/Y}\to I_Z$ is trivial, (d) $I_{X/Y}\subseteq I_{X/Z}$.

(ii) A factoring of $\phi$ through $f$ in (i) is unique in the 2-categorical sense (see Theorem~\ref{univth0}(ii)). In other words, $f$ is a 2-categorical epimorphism.
\end{theorem}
\begin{proof}
The implications (a)$\implies$(b)$\implies$(c)$\Longleftrightarrow$(d) in (i) follow from definitions and the base change property of inertia, see (\ref{inertiaeq1}) in Section \ref{Sec:basic-inertia}. So assume that the map $I_{X/Y}\to I_Z$ is trivial and let us prove (a). Consider a smooth covering of $Y$ by a scheme $Y_0$ and set $Y_1=Y_0\times_YY_0$ and $X_i=Y_i\times_XY$. Since $I_{X_i}=I_X\times_{I_Y}I_{Y_i}$ and $I_{Y_i}$ is trivial, we obtain that $I_{X_i}$ is the pullback of $I_{X/Y}$, and hence the morphisms $I_{X_i}\to I_Z$ are trivial. By Theorem~\ref{univth0}, the morphisms $X_i\to Z$ factor through $Y_i=(X_i)_\cs$ uniquely. We obtain a morphism of groupoids $(Y_1\toto Y_0)\to Z$, which gives rise to a required morphism $Y\to Z$. Part (ii) follows since the morphisms $Y_i\to Z$ are unique up to unique 2-isomorphisms.
\end{proof}

\begin{remark}
(i) The theorem implies that any coarsening morphism $f$ is a 2-categorical epimorphism.

(ii) The implication (c)$\implies$(b) in the theorem is non-trivial. Informally, it indicates that $f^*(I_Y)=I_X/I_{X/Y}$. (To prove that this is indeed a group scheme quotient we should have test it with all group schemes over $X$, while (b) only uses group schemes which are a pullback of some $I_Z$.)
\end{remark}

\subsubsection{Kernel subgroups}
In the sequel, it will be convenient to control a subgroup $I_{X/Y}$ of $I_X$ by geometric points. We suspect this can be done for extensions of tame and \'etale groups, but we restrict to the \'etale case for simplicity. If $G$ is a group scheme over $X$ then by a {\em kernel subgroup} $H\into G$ we mean any subgroup which is the kernel of a homomorphism $G\to G'$.

\begin{lemma}
Assume that $G$ is a finite group scheme over $X$ whose geometric fibers $G_x$ are \'etale. Then any kernel subgroup $H\into G$ is uniquely determined by its geometric fibers $H_x$.
\end{lemma}
\begin{proof}
\'Etale descent and a limit argument reduce this to the case when $X=\Spec(A)$ for a strictly henselian local ring $A$ and $x\in X$ is the closed point. Then $G_x$ is discrete and it is easy to see 
that any kernel subgroup $H$ is the union of the components of $G$ parameterized by the elements of $H_x$, where $H_x$ is a normal subgroup of $G_x$.
\end{proof}

Combining the lemma with Theorem~\ref{univth} we obtain the following result.

\begin{corollary}\label{coarsecor}
Assume that $X$ is DM stack with finite inertia. Then any coarsening morphism $f\:X\to Y$ is uniquely determined by the set of geometric relative stabilizers $(I_{Y/X})_x\into (I_X)_x$, where $x\to X$ runs over the geometric points of $X$.
\end{corollary}

\subsection{Local structure}
It was observed in \cite[Lemma 2.2.3]{Abramovich-Vistoli} that the Keel-Mori theorem allows to describe the local structure of a DM stack $X$ with finite inertia. In the proof of the following theorem we use the same argument to obtain a local description of coarsening morphisms $X\to Y$. Note that a similar result when $X$ is tame but not necessarily DM was proved in \cite[Proposition~3.6]{AOV2}, but that case is more difficult.

\begin{theorem}\label{coarseth}
Assume that $X$ is a DM stack with finite inertia, $f\:X\to Y$ is a coarsening morphism, and $x\to X$ is a geometric point. Let $G_x$ and $G_y$ be the stabilizers of $x$ and $y=f(x)$, let $Z=X_\cs$ and let $z\to Z$ be the image of $x$. Then the homomorphism $\phi\:G_x\to G_y$ is surjective and there exists an \'etale neighborhood $Z'\to Z$ of $z$ such that $Y\times_ZZ'=[(U/H)/G_y]$ and $X\times_ZZ'=[U/G_x]$, where $H=\Ker(\phi)$ and $U$ is an affine scheme acted on by $G_x$ so that the stabilizer at $x$ is $G_x$ itself.
\end{theorem}
\begin{proof}
The idea is to first study the base change with respect to the strict henselization $\oZ=\Spec(\cO_{Z,z})\to Z$, and then approximate $\oZ$ with a fine enough $Z'$.

Set $\oY=\oZ\times_ZY$ and $\oX=\oZ\times_ZX$. Any \'etale covering $V$ of $\oX$ by a scheme contains a connected component $\oU$ mapping surjectively onto $\oZ$. Then $\oU$ is a finite \'etale covering of $\oX$. Furthermore, $\oU$ is strictly henselian and $\oR=\oU\times_\oX\oU$ is a finite \'etale covering of $\oU$, hence it is of the form $\coprod_{i=1}^n\oU$ and the groupoid $\oR\toto\oU$ reduces to an action of a finite discrete group $G$ on $\oU$. Thus $\oX=[\oU/G]$ and comparing the inertia we obtain that $G=G_x$. By the same argument $\oY=[\oW/G_y]$ for a strictly henselian finite $\oZ$-scheme $\oW$, and then $$\oW=(\oW\times_\oY\oX)_\cs=[\oU/H]_\cs=\oU/H.$$

It remains to return from $\oZ$ to an \'etale covering $Z'\to Z$ via standard approximation arguments. Recall that $\oZ$ is the limit of the family of \'etale neighborhoods of $\oz$. We will take $Z'$ to be a fine enough \'etale neighborhood and set $Y'=Z'\times_ZY$ and $X'=Z'\times_ZX$. Choose an \'etale covering of $X$ by a scheme $V$. We showed that $V\times_Z\oZ$ contains a component $\oU$ finite over $\oZ$, hence already some $V\times_ZZ'$ contains a component $U$ finite over $Z'$. In addition, refining $Z'$ we can achieve that $R=U\times_{Z'}U$ is a disjoint union of copies of $U$, and hence $R\toto U$ reduces to an action of $G_x$ and $[U/G_x]=X'$. Refining $Z'$ further we can also achieve that $[(U/H)/G_y]=Y'$.
\end{proof}

\section{Toroidal stacks and moduli spaces}\label{Sec:toroidal}

\subsection{Toroidal schemes}

\subsubsection{References}
The adopt the terminology of \cite{AT1} concerning toroidal schemes and their morphisms with the only difference that we replace Zariski fine logarithmic structures by the \'etale ones. In other words, in this paper we extend the class of toroidal schemes so that it contains ``toroidal schemes with self-intersections" in the terminology of \cite{KKMS}.

Note that when Kato introduced logarithmically regular logarithmic schemes in \cite{Kato-toric}, he worked with Zariski logarithmic schemes for simplicity. However, \'etale locally any fine logarithmic scheme is a Zariski logarithmic scheme, and this allows to easily extend all results about logarithmic regularity to general fs logarithmic schemes, see \cite{Niziol}.

We will make use of Kummer logarithmically \'etale morphisms, see \cite{Niziol-K-theory-of-log-schemes-I} and Section \ref{Sec:Kummer-topology} below.

\subsubsection{Toroidal schemes}
Now, let us recall the main points quickly. In this paper, a {\em toroidal scheme} $X$ is a logarithmically regular logarithmic scheme $(X,M_X)$ in the sense of \cite{Niziol}. Equivalently, one can represent $X$ by a pair $(X,U)$, where the open subscheme $U$ is the locus where the logarithmic structure is trivial. One reconstructs the monoid by $M_X=\cO_{X\et}\cap i_*(\cO^\times_{U\et})$, where $i\:U\into X$ is the open immersion. Usually, we will denote a toroidal scheme $X$ or $(X,U)$.

\subsubsection{Fans}
Recall that the logarithmic stratum $X(n)$ of a logarithmic scheme $(X,\cM_X)$ consists of all points $x\in X$ with ${\rm rank}(\oM_x)=n$. Here and in the sequel we use the convention that $\oM_x$ denotes $\oM_\ox$ for a geometric point $\ox\to X$ over $x$. In particular, $\oM_x$ is defined up to an automorphism, but its rank is well defined.

If $X$ is a toroidal scheme then each stratum $X(n)$ is regular of pure codimension $n$ (essentially, this is the definition of logarithmic regularity). By the {\em fan} of a toroidal scheme $X$ we mean the set $\Fan(X)$ of all generic points of the logarithmic strata of $X$. Also, let $\eta\:X\to\Fan(X)$ denote the contraction map sending a point $x$ to the generic point of the connected component of the logarithmic stratum containing $x$.

\subsubsection{Morphisms}
A morphism of toroidal schemes $(Y,V)\to(X,U)$ is a morphism of the associated logarithmic schemes. Equivalently one can describe it as a morphism $f\:Y\to X$ such that $f(V)\subseteq U$. An important class of morphisms are the logarithmically smooth ones (called {\em toroidal} in \cite{AT1}). Another important class of morphisms are the strict ones: the morphisms that induce an isomorphism $f^*\cM_X\toisom\cM_Y$.

\subsection{Toroidal actions}

\subsubsection{Definitions}
Assume that a diagonalizable group $G$ acts in a relatively affine manner on a toroidal scheme $(X,U)$, see \cite[Sections 5.1, 5.3]{ATLuna}. Recall that the action is {\em simple} at a point $x\in X$ if the stabilizer $G_x$ acts trivially on $\oM_x$, and the action is {\em toroidal} at $x$ if it is simple at $x$ and $G_x=G_{\eta(x)}$. Note that the latter happens if and only if $G_x$ acts trivially on the connected component of the logarithmic stratum through $x$.

\begin{remark}\label{actionrem}
(i) By \cite[Corollary~3.2.18]{AT1}, the set of points $x\in X$, at which the action is toroidal or simple, is open.

(ii) Let us temporary say that the action is quasi-toroidal at $x$ is $G_x=G_{\eta(x)}$. This notion is not so meaningful due to the following examples:

(1) The openness property fails for quasi-toroidality. For example, let $G=\ZZ/2\ZZ$ act on $X=\Spec(k[x,y])$ by switching the coordinates. Then the action is quasi-toroidal at the origin, but it is not quasi-toroidal at other points of the line $X^G$, which is given by $x=y$. Note that this action is not simple at the origin, so the example is consistent with the openness result for the toroidal locus.

(2) Let $G=\ZZ/4\ZZ$ with a generator $g$ act on $X=\Spec(k[x,y])$ by $gx=y$ and $gy=-x$. Then the action is quasi-toroidal everywhere but is not simple at the origin.
\end{remark}

\subsubsection{The groups $G_x^\tor$}\label{torstabsec}
Let $G_{\oM_x}$ be the subgroup of $G_x$ that stabilizes $\oM_x$. By the {\em toroidal stabilizer} at $x$ we mean the subgroup $G_x^\tor=G_{\eta(x)}\cap G_{\oM_x}$ of the stabilizer $G_x$. Clearly, $G_x^\tor$ is the maximal subgroup of $G_x$ that acts toroidally at $x$.

\begin{lemma}\label{Gtorlem}
If a diagonalizable group $G$ acts in a relatively affine manner on a toroidal scheme $X$ then any point $x\in X$ possesses a neighborhood $X'$ such that $G_x^\tor\cap G_{x'}=G_{x'}^\tor$ for any point $x'\in X'$.
\end{lemma}
\begin{proof}
Let $X'$ be obtained by removing from $X$ the Zariski closures of all points $\veps\in\Fan(X)$ which are not generizations of $x$. Thus, $\eta(x')$ is a generization of $\eta(x)$ for any $x'\in X'$. Note that $\oM_{x'}=\oM_{\eta(x')}$ since $\oM_X$ is locally constant along logarithmic strata. Therefore $G_{x'}^\tor=G_{\eta(x')}^\tor$, and it suffices to deal with the case when $x,x'\in\Fan(X)$. Then $x'$ specializes to $x$ and our claim reduces to the check that $G_{\oM_{x}}\cap G_{x'}=G_{\oM_{x'}}$. Since the cospecialization $\phi\:\oM_x\to\oM_{x'}$ is surjective, $G_{\oM_{x}}\cap G_{x'}\subseteq G_{\oM_{x'}}$. Conversely, assume that $g\in G_{\oM_{x'}}\subseteq G_{x'}$ but $g\notin G_{\oM_x}$. Choose any $l\in\oM_x$ not stabilized by $g$ and lift it to an element $t\in\cO_{X,x}$. Then $t\in m_{x'}$ because $g\in G_{x'}$ and hence $\phi(l)\neq 1$. This implies that $g$ acts non-trivially on $\oM_{x'}$, a contradiction.
\end{proof}

\subsubsection{The quotients}
Toroidal stabilizers can also be characterized in terms of the quotient morphisms. To obtain a nice picture we restrict to \'etale groups.

\begin{lemma}\label{quotlem}
Assume that an \emph{\'etale} diagonalizable group $G$ acts in a relatively affine manner on a toroidal scheme $(X,U)$ and $x\in X$ is a point. Then $G_x^\tor$ is the maximal subgroup $H$ of the stabilizer $G_x$ such that if $q\:X\to X/H$ is the quotient morphism then the pair $(X/H,U/H)$ is toroidal at $q(x)$ and the morphism $(X,U)\to(X/H,U/H)$ is Kummer logarithmically \'etale at $x$.
\end{lemma}
\begin{proof}
If $H\subseteq G_x^\tor$, that is $H$ acts toroidally at $x$, then the quotient is as asserted by \cite[Theorem~3.3.12]{AT1}. Conversely, assume that $H$ is such that $q$ is Kummer logarithmically \'etale at $x$. Then $\oM_{q(x)}$ contains $n\oM_x$ for a large enough $n$, and since $G$ acts trivially on $\oM_{q(x)}$, it acts trivially on $\oM_x$. So, the action is simple in a neighborhood of $x$, and we have that $G_x^\tor=G_\eta$. Let $C$ be the connected component of the logarithmic stratum containing $x$. If $H\nsubseteq G_\eta$ then the induced morphism $C\to q(C)$ is ramified at $x$ because $\eta$ is the generic point of $C$. But we assumed that $q$ is logarithmically \'etale, and hence $C\to q(C)$ is \'etale. This shows that $H\subseteq G_\eta=G_x^\tor$, as required.
\end{proof}

\subsubsection{Functoriality}\label{functsec}
Assume that toroidal schemes $X$ and $Y$ are provided with relatively affine actions of diagonalizable groups $G$ and $H$, respectively, $\lambda\:H\to G$ is a homomorphism, and $f\:Y\to X$ is a $\lambda$-equivariant morphism. We want to study when the toroidal inertia groups are functorial in the sense that $H_y^\tor\into \lambda^{-1}(G_x^\tor)$ for any $y\in Y$ with $x=f(y)$. By \cite[Lemma 3.1.6(i)]{AT1}, strict morphisms respect simplicity of the action. The toroidal property is more subtle: the functoriality of toroidal inertia may fail even for surjective fix-point reflecting strict morphisms.

\begin{example}\label{nonexam}
Let $X=\Spec(k[x,y])$ with the toroidal structure $(x)$ and $G=\ZZ/2\ZZ$ acting by the sign both on $x$ and $y$. Then the action is not toroidal at the origin $O$, so $G_{X,O}^\tor=1$. Let $Y$ be the $x$-axis $\Spec(k[x])$ with the toroidal structure $(x)$. Then $Y$ embeds $G$-equivariantly into $X$, but the action is toroidal on $Y$ and hence $G_{Y,O}^\tor=G$ is not mapped into $G_{X,O}^\tor$. Furthermore, if $X_0=X\setminus\{O\}$ then $X_0\coprod Y\to X$ is a surjective fix-point reflecting strict morphism which is not functorial for the toroidal inertia.
\end{example}

\begin{remark} As this example shows, the statement in \cite[Lemma 3.1.9(ii)]{AT1} needs to be corrected to read ``and the converse is true if $f$ is \emph{\'etale} and surjective", and the proof should  read ``Hence (ii) follows from (i), Lemma 3.1.6(i) and \emph{\'etale descent"}. This does not affect other results of that paper, since only the direct implication was used. Still, for the inverse implication the \'etale assumption can be weakened. For example, it suffices to assume that the morphism respects the fans (see below).
\end{remark}

The problem in Example \ref{nonexam} is that $O$ is in the fan of $Y$ but not in the fan of $X$, and the stabilizer drops at $\eta_X(O)$. This motivates the following definition: a $\lambda$-equivariant morphism $f\:Y\to X$ is said to {\em respect the fans} if $f(\Fan(Y))\subseteq\Fan(X)$. Also, we say that $f$ is {\em injective on monoids} if the homomorphisms $\oM_{f(y)}\to\oM_y$ are injective.

\begin{lemma}\label{functlem0}
Let $f\:Y\to X$ be a $\lambda$-equivariant morphism as in \S\ref{functsec}, and let $y\in Y$ be a point with $x=f(y)$ and the induced homomorphism $\lambda_y\:H_y\to G_x$.

(i) Assume that $f$ respects the fans and is injective on monoids locally at $y$ (for instance, this is satisfied when $f$ is logarithmically flat at $y$). Then $H_y^\tor\into \lambda_y ^{-1}(G_x^\tor)$.

(ii) Assume that $f$ is fix-point reflecting and one of the following conditions holds: (a) $f$ is strict at $y$, (b) $f$ respects the fans and the action of $H$ is simple at $y$. Then $\lambda_y^{-1}(G_x^\tor)\into H_y^\tor$.
\end{lemma}
\begin{proof}
(i) By our assumption, $f(\eta_Y(y))=\eta_X(x)$ and hence $H_{\eta_Y(y)}\into \lambda_y ^{-1}(G_{\eta_X(x)})$. Also, since $\oM_x\into\oM_y$ we have that $H_{\oM_y}\into \lambda_y ^{-1}(G_{\oM_x})$. The two inclusions imply that $H_y^\tor\into \lambda_y ^{-1}(G_x^\tor)$.

(ii) We can identify $G_x$ and $H_y$ via $\lambda_y$. In case (a), $G_{\oM_x}=H_{\oM_y}$. If $C_x$ and $C_y$ are the connected components of the logarithmic strata through $x$ and $y$ then $f(C_y)\subseteq C_x$, and hence $G_{\eta_X(x)}\into G_{f(\eta_Y(y))}=H_{\eta_Y(y)}$. The claim follows.

In case (b), we should prove that $G_x^\tor$ acts toroidally at $y$. Since the action is simple at $y$, it suffices to check that $G_x^\tor\into G_{\eta_Y(y)}$. By our assumptions, $f(\eta_Y(y))=\eta_X(x)$ and hence $G_{\eta_Y(y)}=G_{\eta_X(x)}$. But $G_x^\tor\into G_{\eta_X(x)}$ by definition.

\end{proof}

\begin{remark}\label{functrem}
(i) It is easy to see that logarithmically smooth (even logarithmically flat, see \cite[Definition~2.1]{Niziol-K-theory-of-log-schemes-I}) morphisms of fs logarithmic schemes respect fans and are injective on monoids.

(ii) The assumption that the action is simple in Lemma~\ref{functlem0}(ii)(b) certainly simplifies the proof, but we do not know if it is necessary.
\end{remark}

\subsubsection{Toroidal inertia}\label{Sec:tor-inertia-quotient}
For the sake of completeness we note that the groups $G_x^\tor$ glue to a {\em toroidal inertia} group scheme $I_X^\tor$ over the $G$-scheme $X$. Namely, if $\oveps$ denotes the Zariski closure of $\veps$ then $$I_X^\tor:=\cup_{\veps\in\Fan(X)}\ G_\veps^\tor\times \oveps$$ is a subgroup of $G\times X$, which is obviously contained in $I_X$. Since $G$ is discrete there is no ambiguity about the scheme structure: $G\times X=\coprod_{g\in G}X$ and $I_X^\tor=\coprod_{g\in G}X^g$, where $X^g$ is the closed subscheme fixed by $g$. The functoriality results of Lemma~\ref{functlem0} extend to the toroidal inertia schemes in the obvious way

\subsection{Toroidal stacks}
Using descent, the notions of toroidal schemes and morphisms can be easily extended to Artin stacks, see \cite[Section~5]{Olsson-logarithmic}. We will stick to the case of DM stacks, since only they show up in our applications. A minor advantage of this restriction is that one can work with the \'etale topology instead of the lisse-\'etale one.

\subsubsection{Logarithmic structures on stacks}
By a logarithmic structure on an DM stack $X$ we mean a sheaf of monoids $M_X$ on the \'etale site $X\et$ and a homomorphism $\alp_X\:M_X\to\cO_{X\et}$ inducing an isomorphism $M^\times_X\toisom\cO^\times_{X\et}$. If $p_{1,2}\:X_1\toto X_0$ is an atlas of $X$ then giving a logarithmic structure $M$ is equivalent to giving compatible logarithmic structures $M_{X_i}$ in the sense that $p_i^{-1}M_{X_0}=M_{X_1}$ for $i=1,2$. We say that $(X,M_X)$ is fine, saturated, etc., if $(X_0,M_{X_0})$ is so. We use here that these properties of $M_{X_0}$ are \'etale local on $X_0$, and hence are independent of the choice of the atlas.

\subsubsection{Logarithmic stacks and atlases}
By a logarithmic stack $(X,M_X)$ we mean a stack provided with a logarithmic structure. In this case, for any smooth atlas $X_1\toto X_0$ of $X$ we provide $X_0$ and $X_1$ with the pullbacks of $M_X$ and say that $(X_1,M_{X_1})\toto(X_0,M_{X_0})$ is an atlas of $(X,M_X)$. Indeed, $\alp_X\:M_X\to\cO_{X\et}$ is uniquely determined by this datum.

\subsubsection{Toroidal stacks}
A logarithmic stack $(X,M_X)$ is {\em logarithmically regular} or {\em toroidal} if it admits an atlas such that $(U,M_U)$ is toroidal. In this case any atlas is toroidal because logarithmic regularity is a smooth-local property, see \cite[Proposition~7.5.46]{Gabber-Ramero}.

Furthermore, the triviality loci $U_i\subseteq X_i$ of $M_{X_i}$ are compatible with respect to the strict morphisms $p_{1,2}$, hence $U_0$ descends to an open substack $i\:U\into X$ that we call the triviality locus of $M_X$. Furthermore, when $(X,M_X)$ is logarithmically regular, $U$ determines the logarithmic structure by $M_X=\cO_{X\et}\cap i_*(\cO^\times_{U\et})$ because the same formulas reconstruct $M_{X_i}$. In the sequel, we will often view toroidal stacks as pairs $(X,U)$. Again, a morphisms $(Y,V)\to(X,U)$ of toroidal stacks is nothing else but a morphism $f\:Y\to X$ of stacks such that $V\into f^{-1}(U)$.

\subsection{Total toroidal coarsening}\label{torcoarsesec}
Let $(X,U)$ be a toroidal DM stack.

\subsubsection{Toroidal coarsening morphisms}
We say that a coarsening morphism $f\:X\to Y$ is {\em toroidal} if $f(|U|)$ underlies an open substack $V\into Y$, the pair $(Y,V)$ is a toroidal stack, and the morphism $(X,U)\to(Y,V)$ is Kummer logarithmically \'etale. If exists, the final object of the category of toroidal coarsening morphisms of $X$ will be called the {\em total toroidal coarsening} of $X$ and denoted $\phi_X\:X\to X_\tcs$.

Our next goal is to construct $X_\tcs$. By Corollary~\ref{coarsecor}, $\phi_X$ is determined by the geometric points of its inertia, so our plan is as follows. First, we will extend the notion of toroidal stabilizers from \S\ref{torstabsec} to geometric points of stacks, and then we will use them to construct $\phi_X$ so that, indeed, $(I_{\phi_X})_x$ is the toroidal stabilizer of $x$. In this context, $I_{\phi_X}$ is the generalization to toroidal stacks of the toroidal inertia $I_X^\tor$ from \S\ref{Sec:tor-inertia-quotient}.

\subsubsection{Toroidal inertia}\label{toroidalinertia}
Let $Z=X_\cs$. By Theorem~\ref{coarseth}, a geometric point $x\to X$ possesses a representable \'etale neighborhood $X'=X\times_ZZ'$ of the form $[X'_0/G_x]$. Pulling back the toroidal structure of $X$ we obtain a $G_x$-equivariant toroidal structure on $X'_0$ and we take $G_{X'_0,x}^\tor$ to be the maximal subgroup of $G_x$ acting toroidally along $x$. By the following lemma, we can denote this group simply $G_x^\tor$. It will be called the {\em toroidal stabilizer} at $x$. Note also that $\oM_{X,x}=\oM_{X'_0,x}$, and hence we obtain an action of $G_x$ on $\oM_x$.

\begin{lemma}\label{Gtorstacklem}
With the above notation, the group $G_{X'_0,x}^\tor$ and the action of $G_x$ on $\oM_x$ are independent of the choice of quotient presentation $X'=[X'_0/G_x]$.
\end{lemma}
\begin{proof}
Given a finer \'etale neighborhood $Z''\to Z'$ of the image of $x$ in $Z$, set $X''=X\times_ZZ''$ and $X''_0=X'_0\times_{X'}X''$. In particular, $X''=[X''_0/G_x]$. It suffices to check that $G_{X'_0,x}^\tor=G_{X''_0,x}^\tor$. Being a base change of a morphism of algebraic spaces, the morphism $X''\to X'$ is inert, and it follows that the strict \'etale $G_x$-equivariant morphism $X''_0\to X'_0$ is inert. Therefore, $G_{X'_0,x}^\tor=G_{X''_0,x}^\tor$ by Lemma~\ref{functlem0}. Also, it is clear that $\oM_{X'_0,x}=\oM_{X''_0,x}$ as $G_x$-sets.
\end{proof}

Functoriality properties from Lemma~\ref{functlem0} extend to stacks straightforwardly.

\begin{lemma}\label{functlem}
Assume that $f\:Y\to X$ is a morphism of toroidal stacks, $y\to Y$ is a point, $x=f(y)$, and $\lambda_y\:G_y\to G_x$ is the induced homomorphism. Then,

(i) If $f$ is logarithmically flat at $y$ then $G_y^\tor\subseteq \lambda_y ^{-1}(G_x^\tor)$.

(ii) If $f$ is logarithmically flat and inert at $y$ and $G_y$ acts trivially on $\oM_y$ then $G_y^\tor=\lambda_y^{-1}(G_x^\tor)$.
\end{lemma}
\begin{proof}
If $Y=[Y_0/G_y]$ and $X=[X_0/G_x]$ are quotients of affine schemes then the toroidal stabilizers equal to the toroidal stabilizers of the actions of $G_y$ and $G_x$ on $Y_0$ and $X_0$, respectively. Hence the assertion follows from Lemma~\ref{functlem0}. The general case is reduced to this by local work on the coarse moduli spaces: first we base change both schemes with respect to an \'etale morphism $Z'\to X_\cs$ such that $X$ becomes as required. Then we replace $Y$ further by an appropriate \'etale neighborhood of $y$ induced from  an \'etale neighborhood of its image in $Y_\cs$.
\end{proof}

\subsubsection{Toroidal orbifolds}\label{orbsec}
In the sequel, by a {\em toroidal orbifold} we mean a toroidal DM stack $X$ with finite diagonalizable inertia. We say that $X$ is {\em simple} if for any point $y\to Y$ the group $G_y$ acts on $\oM_y$ trivially.

\subsubsection{The construction}
Now we can construct the total toroidal coarsening. 

\begin{theorem}\label{modulith}
Let $\cC$ be the 2-category of toroidal orbifolds and let $X$ be an object of $\cC$. Then,

(i) The total toroidal coarsening $X_\tcs$ exists.

(ii) For any geometric point $x\to X$, we have $(I_{X/X_\tcs})_x = G_x^\tor$, where $(I_{X/X_\tcs})_x$ is the relative stabilizer and $G_x^\tor$ the toroidal inertia group.

(iii) Any logarithmically flat morphism $h\:Y\to X$ in $\cC$ induces a morphism $h_\tcs\:Y_\tcs\to X_\tcs$ with a 2-commutative diagram
$$
\xymatrix{
Y\ar[d]_h\ar[r]^{\phi_Y} & Y_\tcs\ar[d]^{h_\tcs} \\
X\ar[r]_{\phi_X}\ar@{=>}[ur]^\alpha & X_\tcs
}
$$
and the pair $(h_\tcs,\alpha)$ is unique in the 2-categorical sense: if $(h'_\tcs,\alpha')$ is another such pair then there exists a unique 2-isomorphism $h'_\tcs=h_\tcs$ making the whole diagram 2-commutative.

(iv) Assume that $h$ is logarithmically flat and inert, and $Y$ is simple. Then the diagram in (iii) is 2-cartesian.
\end{theorem}
\begin{proof}
We will first construct a coarsening $X\to\tX$ that satisfies (ii); by Corollary \ref{coarsecor} this determines $\tX$ uniquely.

Assume first that $X$ satisfies the following two conditions: (a) $X=[W/G]$, where $W$ is a quasi-affine toroidal scheme and $G$ is an \'etale diagonalizable group, (b) $G$ contains a subgroup $G^\tor$ such that $G_w^\tor=G^\tor\cap G_w$ for any $w\in W$. Then $\tX=[(W/G^\tor)/(G/G^\tor)]$ is a coarsening of $X$ satisfying (ii). Let us deduce the general case.

Theorem~\ref{coarseth} implies that for any geometric point $x\to X$ we can find a neighborhood $Z_0\to Z:=X_\cs$ such that $X_0=X\times_ZZ_0$ satisfies (a) with $G=G_x$, say $X_0=[W_0/G_x]$. Since $|X_0|=|Z_0|$, by Lemma~\ref{Gtorlem} we can shrink $Z_0$ so that condition (b) is satisfied with $G^\tor=G_x^\tor$. In particular, we obtain a coarsening $X_0\to\tX_0$ satisfying (ii). Set $Z_1=Z_0\times_ZZ_0$ and $X_1=X\times_ZZ_1$. Then both $\tX_1:=\tX_0\times_{Z_0,p_1}Z_1$ and $\tX_0\times_{Z_0,p_2}Z_1$ are coarsenings of $X_1$ that satisfy (ii), hence they coincide and we obtain the following diagram
$$
\xymatrix{
X_1\ar@<0.5ex>[d]\ar@<-0.5ex>[d]\ar[r] & \tX_1\ar@<0.5ex>[d]\ar@<-0.5ex>[d]\ar[r] & Z_1\ar@<0.5ex>[d]\ar@<-0.5ex>[d]\\
X_0\ar[r]\ar[d] & \tX_0\ar[r] & Z_0\ar[d]\\
X\ar[rr] &  & Z
}
$$
Moreover, the same argument with $X_2,\tX_2$ and $Z_2$ provides $\tX_1\toto \tX_0$ with a structure of an inert \'etale groupoid. Therefore, the quotient $\tX=[\tX_0/\tX_1]$ exists by Lemma~\ref{inertquot} and the sequence $X_0\to \tX_0\to Z_0$ is the base change of a sequence $X\to \tX\to Z$. It follows that $X\to \tX$ is a coarsening morphism, and since $\tX_0\to \tX$ is inert, $\tX$ has stabilizers as required by (ii).

Now, let us check (i). It suffices to show that a coarsening $f\:X\to T$ is toroidal in an open neighborhood of a geometric point $x\to X$ if and only if the normal subgroup $H=(I_{X/T})_x$ of $G_x$ is contained in $G_x^\tor$. It suffices to check this claim \'etale locally, hence the same argument as earlier reduces the check to the case when $X$ satisfies (a) and (b), in particular, $X=[W/G_x]$ and $T=[(W/H)/(G/H)]$. \'Etale locally $f$ is of the form $W\to W/H$, hence the assertion follows from Lemma~\ref{quotlem}.

To prove (iii) we should prove that the morphism $Y\to X_\tcs$ factors through $Y_\tcs$ uniquely. So, by Theorem~\ref{univth} we should prove that $I_{Y/Y_\tcs}$ is mapped to zero in $I_{X_\tcs}$. We claim that, moreover, the map $I_Y\to I_X$ takes $I_{Y/Y_\tcs}$ to $I_{X/X_\tcs}$. It suffices to check this on the geometric points, since the inertia are \'etale for DM stacks. But the latter is precisely Lemma~\ref{functlem}(i).

Let us prove (iv). Let $Q$ denote the square diagram from (iii). Choose an \'etale covering of $f\:Z\to Y_\tcs$ with $Z$ a scheme. It suffices to show that the base change square $f^*(Q):=Q\times_{X_\tcs}Z$ is 2-cartesian. For any point $y\to Y$ with $x=h(y)$ we have that $G^\tor_y\toisom G_x^\tor$ by Lemma~\ref{functlem}(ii). Hence $I_{\phi_Y(y)}=I_{\phi_X(x)}$, and we obtain that the morphism $h_\tcs$ is inert. It follows that $Z\times_{X_\tcs}Y_\tcs$ is an algebraic space. Thus, the morphisms $f^*(\phi_X)$ and $f^*(\phi_Y)$ are coarsening morphisms whose targets are algebraic spaces, and hence both are usual coarse spaces. Since coarse spaces are compatible with arbitrary base changes in the case of DM stacks, the square $f^*(Q)$ is 2-cartesian.
\end{proof}

\subsubsection{Relative toroidal coarsening} The fact that the total toroidal coarsening behaves well under pullbacks allows us to define a relative version. Consider a toroidal morphism of algebraic stacks $X \to S$, where $X$ is a toroidal orbifold as before. The relative toroidal coarsening of $X$ over $S$, if exists, is the final object $X \to X_{\tcs /S} \to S$ among factorizations $X \to Z \to S$, where $X \to Z$ is a toroidal coarsening morphism relative to $S$.

\begin{corollary}\label{modulicor} 	The relative toroidal coarsening exists. For any geometric point $x\to X$, we have $(I_{X/X_{\tcs/S}})_x = G_x^{\tor/S}$, where $(I_{X/X_{\tcs/S}})_x$ is the relative stabilizer and $G_x^{\tor/S}$ the toroidal inertia group relative to $S$.
\end{corollary}
\begin{proof}
Choose a presentation $S_1 \double S_0$ of $S$ and let $X_1 \double X_0$ be the pullback diagram, giving an inert presentation of $X$. Setting $\tilX_i=(X_i)_\tcs$ we obtain a diagram of inert groupoids
$$
\xymatrix{
X_1\ar@<0.5ex>[d]\ar@<-0.5ex>[d]\ar[r] & \tX_1\ar@<0.5ex>[d]\ar@<-0.5ex>[d]\ar[r] & S_1\ar@<0.5ex>[d]\ar@<-0.5ex>[d]\\
X_0\ar[r]\ar[d] & \tX_0\ar[r]\ar@{.>}[d] & S_0\ar[d]\\
X\ar@{.>}[r] & \tX\ar@{.>}[r]  & S,
}
$$
implying the existence of $X \to \tX \to S$ as before.
\end{proof}

\section{Destackification}\label{destacksec}

\subsection{The main result}

\subsubsection{Blowings up of toroidal stacks}\label{blowtorsec}
We say that a morphism $f\:(X',U')\to(X,U)$ of toroidal stacks is the {\em blowing up along} a closed substack $Z\into X$ if $f\:X'\to X$ is a blowing up along $Z$ and $U'=f^{-1}(U)\setminus f^{-1}(Z)$. For example, a blowing up of toroidal schemes is a blowing up of usual schemes $f\:X'\to X$ such that the toroidal divisor $X'\setminus U'$ of $(X',U')$ is the union of the preimage of the toroidal divisor of $(X,U)$ and the exceptional divisor of $f$. We use the same definition for normalized blowings up.

\subsubsection{Torification}
Our destackification results are based on and can be viewed as stack-theoretic enhancements of torification theorems of \cite{AT1}. In appendix \ref{torapp} we recall these results and slightly upgrade them according to the needs of this paper.

\subsubsection{Destackification theorem}
Let us first formulate our main results on destackification. Their proof will occupy the rest of Section~\ref{destacksec}. Using the torification functors $\cT$ and $\tilcT$ we will construct two destackification functors: $\cF$ and $\tilcF$. The former one has stronger functoriality properties, but only applies to toroidal stacks with inertia acting \emph{simply}.

\begin{theorem}\label{destackth}
Let $\tilcC$ be the category of toroidal orbifolds.

(i) For any object $X$ of $\tilcC$ there exists a sequence of blowings up of toroidal stacks $\tilcF_X\:X_n\dashrightarrow X$ such that $(X_n)_\tcs=(X_n)_\cs$.

(ii) In addition, one can choose $\tilcF$ compatible with surjective strict inert morphisms $f\:X'\to X$ from $\tilcC$ in the sense that for any such $f$ the sequence $\tilcF_{X'}$ is the pullback of $\tilcF_X$ with empty blowings up omitted.
\end{theorem}

\begin{theorem}\label{destacksimpleth}
Let $\cC$ be the category of simple toroidal orbifolds. Then to any object $X$ in $\cC$ one can associate a blowing up of toroidal stacks $\cF_X\:X_1\to X$ along an ideal $I_X$ and a blowing up $\cF^0_X\:X_0\to X_\cs$ along an ideal $J_X$ so that

(i) $(X_1)_\tcs=(X_1)_\cs=X_0$.

(ii) If $f\:X'\to X$ is a surjective inert morphism in $\cC$, which is either strict or logarithmically smooth, then $\cF_{X'}$ and $\cF^0_{X'}$ are the pullbacks of $\cF_X$ and $\cF^0_{X}$, respectively.
\end{theorem}

\subsection{The proof}
We will work with Theorem~\ref{destacksimpleth} for concreteness. The proof of Theorem~\ref{destackth} is similar and involves less details; the main difference is that  one should use Theorem~\ref{lambdath} as the torification input instead of Theorem~\ref{nostrictth} and Corollary~\ref{nostrictcor}.

We will construct the functor $\cF$ by showing that the torification functor $\cT$ descends to stacks. This will be done in two stages: first we will show its descent to global quotients $[W/G]$ and then will use \'etale descent with respect to inert morphisms.

\subsubsection{Step 1: the global quotient case}\label{globalquotsec}
We will first prove the theorem for the subcategory $\cC'$ of $\cC$ whose objects $X$ are of the form $[W/G]$, where $G$ is an \'etale diagonalizable group acting on a toroidal quasi-affine scheme $W$. Since the blowing up and the center of $\cF_{W,G}$ are $G$-equivariant, they descend to $X$. Namely, there exists a unique blowing up of toroidal stacks $\cT_{X,W}\:X_1\to X$ whose pullback to $W$ is $\cF_{W,G}\:W_1\to W$. Since $[W/G]_\cs=W/G$, we simply set $\cF^0_{X,W}=\cT^0_{W,G}$. We claim that these $\cF_{W,G}$ and $\cF^0_{W,G}$ are independent of the choice of the covering $W$.

Suppose that $X=[W'/G']$ is another such representation. Note that $X=[W\times_XW'/G\times G']$ and hence it suffices to compare the blowings up induced from $W$ and $W\times_XW'$. In other words, we can assume that $G=G'/H$ and $X=X'/H$, where $H$ acts freely on $X'$. In this case the projection $X'\to X$ is inert and $\lambda$-equivariant for the projection $\lambda\:G'\onto G$. Thus, $\cF_{X',G'}$ and $\cF^0_{X',G'}$ are the pullbacks of $\cF_{X,G}$ and $\cF^0_{X,G}$ by Theorem~\ref{lambdath}. In particular, $\cF^0_{X',G'}=\cF^0_{X,G}$. This proves that $\cF_{W,G}=\cF_{W',G'}$ and $\cF^0_{W,G}=\cF^0_{W',G'}$, and in the sequel we can safely write $\cF_W$ and $\cF^0_W$.

The properties of $\cF$ and $\cF^0$ are checked similarly, so we will only discuss $\cF$. The action of $G$ on $W_1$ is toroidal, hence $G_w=G_w^\tor$ for any $w\in W_1$. Since $X_1=[W_1/G]$, the definition of toroidal stabilizers in \S\ref{toroidalinertia} implies that $G_x=G_x^\tor$ for any geometric point $x\to X$. Therefore, $(X_1)_\tcs=(X_1)_\cs$ by Theorem~\ref{modulith}. Assume that $f\:X'\to X$ is a strict inert morphism in $\cC'$. Choose presentations $X=[W/G]$ and $X'=[W'/G']$. Replacing the latter presentation by $[W'\times_XW/G\times G']$, we can assume that there is a homomorphism $\lambda\:G'\to G$ such that $f$ lifts to a $\lambda$-equivariant morphism $h\:W'\to W$. Since $f$ is inert, the same is true for $h$, and hence $\cF_W$ and $\cF_{W'}$ are compatible. By the definition of $\cF$ on $\cC'$, we obtain that $\cF_X$ and $\cF_{X'}$ are compatible too.

\subsubsection{Step 2: inert \'etale descent}
Assume now that $X$ is an arbitrary toroidal orbifold. By Theorem~\ref{coarseth} the coarse moduli space $Z=X_\cs$ possesses an \'etale covering $Z'=\coprod_{i=1}^lZ_i\to Z$ such that each $Z_i$ is affine and each $X_i=X\times_ZZ_i$ lies in $\cC'$, say $X_i=[W_i/G_i]$. Note that $X'=\coprod_{i=1}^lX_i$ is also in $\cC'$, for example, $X'=W'/G'$ for $W'=\coprod_i(X_i\times\prod_{j\neq i}G_j)$ and $G'=\prod_jG_j$. Furthermore, $X''=X'\times_XX'$ is also in $\cC'$ since $X''=[W''/G'']$ for $W''=W'\times_XW'$ and $G''=G'\times G'$. (Although $I_X\to X$ is finite, $X$ does not have to be separated, so $W''$ can be quasi-affine even though we started with an affine $W'$.)

By \S\ref{globalquotsec} $\cF$ was defined for $X'$ and $X''$ and $\cF_{X''}$ is the pullback of $\cF_{X'}$ with respect to either of the projections $X''\toto X'$. By \'etale descent, $\cF_{X'}$ is the pullback of a blowing up $\cF_{X,X'}\:X_1\to X$ of the toroidal stack $X$. In the same fashion, the blowings up $\cF^0_{X'}$ and $\cF^0_{X''}$ of $Z'$ and $Z''=Z'\times_ZZ'$ descend to a blowing up $\cF^0_{X,X'}\:Z_1\to Z$, and by descent $(X_1)_\cs=Z_1$. Independence of the covering $X'\to X$ is proved as usually: given another such covering one passes to a their fiber product, which is also a global quotient of a quasi-affine scheme, and then uses that $\cF$ is compatible with inert morphisms.

We have now constructed $\cF_X$ and $\cF^0_X$ for an arbitrary object of $\cC$. Their properties are established by \'etale descent via a covering $f\:X'\to X$ as above. For example, for any geometric point $x\to X_1$ choose a lifting $x'\to X'_1$. Then $G_x=G_{x'}$ because $f$ is inert, and hence $f_1\:X'_1\to X_1$ is inert too. In addition, $G_x^\tor=G_{x'}^\tor$ by Lemma~\ref{functlem}, and $G_{x'}=G_{x'}^\tor$ by Step 1. Thus, $G_x=G_x^\tor$, and hence $(X_1)_\tcs=(X_1)_\cs$.


\section{Kummer blowings up}

\subsection{Permissible centers}

\subsubsection{Toroidal subschemes}
Let $X$ be a toroidal scheme. We say that a closed subscheme $Y$ of $X$ is {\em toroidal} if $(Y,\cM_X|_Y)$ is toroidal. Thus toroidal closed subschemes correspond to strict closed immersions of toroidal schemes.

\begin{lemma}\label{chartlem}
Let $X$ be a toroidal scheme and $Y$ a closed subscheme of $X$. Then $Y$ underlies a toroidal subscheme if and only if locally at any point $y\in Y$ there exist regular coordinates $t_1\.t_n\in\cO_{X,y}$ and $m\le n$ such that $Y=V(t_1\.t_m)$ locally at $y$.
\end{lemma}
\begin{proof}
The inverse implication follows from the formal-local description of toroidal schemes. Assume that $Y$ is toroidal and let us construct required coordinates at $y$. We can assume that $X$ and $Y$ are local with closed point $y$. Let $d$ be the rank of $\oM_{X,y}=\oM_{Y,y}$, and let $n$ and $n-m$ be the dimensions of the closed logarithmic strata $X(d)$ and $Y(d)$. Since $X(d)$ and $Y(d)$ are regular, $\cO_{X(d),y}$ possesses a regular family of parameters $t'_1\.t'_n$ such that $V(t'_1\.t'_m)=Y(d)$. Lift them to coordinates $t_1\.t_n\in\cO_{X,y}$. Since $Y(d)=X(d)\times_X Y$, we can also achieve that $t_1\.t_m$ vanish on $Y$. The scheme $V(t_1\.t_m)$ is integral (even toroidal) by the inverse implication, and $\dim(X)=d+n$ and $\dim(Y)=d+n-m$, hence the closed immersion $Y\into V(t_1\.t_m)$ is an isomorphism.
\end{proof}

\subsubsection{Permissible centers}\label{centersec}
Let $X$ be a toroidal scheme. A closed subscheme $Z=\Spec_X(\cO_X/I)$ is called a {\em permissible center} if locally at any point $z\in Z$ it is the intersection of a toroidal subscheme and a monomial subscheme, that is, there exists a regular family of parameters $t_1\.t_n$ and a monomial ideal $J$ such that $I=(t_1\.t_l,J)$ for $l\le n$.

\subsubsection{Playing with the toroidal structure}
A standard method used in toroidal geometry is to enlarge/decrease the toroidal structure by adding/removing components to/from $X\setminus U$. For example, see \cite[\S\S3.4,3.5]{AT1}. We will use this method, and here is a first step.

\begin{lemma}\label{increaselem}
Assume that $(X,U)$ is a local toroidal scheme, $C$ is the closed logarithmic stratum and $t_1\.t_n$ a regular family of parameters of $\cO_{C,x}$. Let $W$ be obtained from $U$ by removing the divisors $V(t_1)\.V(t_l)$, where $0\le l\le n$. Then $(X,W)$ is toroidal and $\oM_{(X,U'),x}=\oM_{(X,U),x}\oplus\NN^l$.
\end{lemma}
\begin{proof}
The equality of the monoids is clear. Since the intersection of $C$ with $V(t_1\.t_l)$ is regular of codimension $l$ we obtain that $(X.W)$ is toroidal at $x$ and hence toroidal.
\end{proof}

\begin{corollary}\label{increasecor}
Assume that $(X,U)$ is a toroidal scheme and $Z\into X$ is a permissible center. Then locally on $X$ one can enlarge the toroidal structure of $X$ so that $Z$ is a monomial subscheme of the new toroidal scheme $(X,W)$.
\end{corollary}
\begin{proof}
Locally at $x\in X$ the center is given by $(t_1\.t_l,J)$, where $J$ is toroidal and of the logarithmic stratum $C$ through $x$. Set $W=U\setminus \cup_{i=1}^l V(t_i)$
and use Lemma~\ref{increaselem}.
\end{proof}

\subsubsection{Functoriality}
Permissible centers are respected by logarithmically smooth morphisms.

\begin{lemma}\label{centersfunct}
Assume that $f\:Y\to X$ is a logarithmically smooth morphism of toroidal schemes and $Z\into X$ is a permissible center (resp. a toroidal subscheme). Then $Z\times_XY$ is a permissible center (resp. a toroidal subscheme) in $Y$.
\end{lemma}
\begin{proof}
Note that $f$ induces smooth morphisms between logarithmic strata of $Y$ and $X$. It follows that if $t_1\.t_n$ are regular coordinates at $x\in X$ then their pullbacks form a part of a family of regular coordinates at a point $y\in f^{-1}(x)$. In view of Lemma~\ref{chartlem}, this implies the claim about toroidal subschemes. Since pullback of a monomial subscheme is obviously monomial, we also obtain the claim about permissible centers.
\end{proof}

\subsection{Permissible blowings up}

\subsubsection{The model case}\label{blowupsec}
We will prove that permissible centers give rise to normalized blowings up of toroidal schemes in the sense of \S\ref{blowtorsec}. This can be done very explicitly in the model case when $X=\bfA_M^n=\Spec(B[M,t_1\.t_n])$, where $B$ is an arbitrary regular ring, $M$ is a toric monoid, and $I=(t_1\.t_n,m_1\.m_r)$ for $m_i\in M$. For the sake of illustration we consider this case separately. Let $X'=Bl_I(X)^\nor$ be the normalized blowing up of $X$ along $I$. We have two types of charts:

(1) The $t_i$-chart is $\bfA^{n-1}_N=\Spec(B[N,\frac{t_1}{t_i}\.\frac{t_n}{t_i}])$, where $N$ is the saturation of the submonoid of $M\oplus\ZZ t_i$ generated by $M$, $t_i$ and the elements $m_1-t_i\.m_r-t_i$. In particular, for any point $x'$ of the chart with image $x\in X$ one has that $\rk(\oM_{y'}\le\rk(\oM_x)+1$. The monoid $N$ is still sharp.

(2) The $m_j$-chart is $\bfA^{n}_P=\Spec(B[P,\frac{t_1}{m_j}\.\frac{t_n}{m_j}])$, where $P$ is the saturation of the submonoid of $M^\gp$ generated by $M$ and the elements $m_1-m_j\.m_r-m_j$. In particular, the rank does not increase on this chart: $\rk(\oM_{x'})\le\rk(\oM_x)$ for any point $x'$ sitting over $x\in X$. The monoid $P$ may not be sharp.

\subsubsection{The general case}
One can deal with the general case similarly by reducing to formal charts, but this is slightly technical, especially in the mixed characteristic case. A faster way is to play with the toroidal structure, reducing to the known properties of toroidal blowings up.

\begin{lemma}\label{permislem}
Assume that $(X,U)$ is a toroidal scheme and $f\:X'\to X$ is the normalized blowing up along a permissible center $Z\into X$, and set $U'=f^{-1}(U\setminus Z)$. Then $(X',U')$ is a toroidal scheme and hence $f$ underlies a normalized blowing up of toroidal schemes.
\end{lemma}
\begin{proof}
The question is local on $X$, so we can assume that $X=\Spec(A)$ is a local scheme with closed point $x$. Then $Z=V(t_1\.t_l,m_1\.m_r)$, where $m_i$ are monomials and $t_1\.t_n$ is a family of regular parameters of the logarithmic stratum through $x$. Set $W=U\setminus\cup_{i=1}^l V(t_i)$, then $(X,W)$ is toroidal by Lemma~\ref{increaselem} and $Z$ is a monomial subscheme of $(X,W)$. Set $W'=f^{-1}(W\setminus Z)$, then $(X',W')$ is toroidal and the toroidal blowing up $(X',W')\to(X,W)$ is a toroidal morphism, see \cite[Section~4]{Niziol} for proofs or \cite[Lemma~4.3.3]{AT1} for a summary. Furthermore, $X'\setminus U'$ is obtained from $X'\setminus W'$ by removing the strict transforms $D'_i$ of $D_i=V(t_i)$, so we should prove that this operation preserves the toroidal property. By \cite[Theorem~2.3.15]{AT1} it suffices to prove that each $D'_i$ is a Cartier divisor.

Now choose $y\in(t_1\.t_l,m_1\.m_r)$ and let study the situation on the $y$-chart $X'_y$. We claim that the inclusion $D'_i|_{X'_y}\into V(\frac{t_i}y)$ is an equality and hence $D'_i$ is Cartier, as required. If $y=t_i$ there is nothing to prove, so assume that $y\neq t_i$. It suffices to show that $V(\frac{t_i}y)$ is integral. So, for any $x'\in X'_y$ it suffices to prove that $\oM_{x'}$ splits as $Q\oplus (t_i-y)\NN$. To compute $\oM_{x'}$ we recall that toroidal blowings up are base changes of toric blowings up of the charts. In particular, $X'\to X$ is the base change of the blowing up of $\Spec(\ZZ[M,t_1\.t_l])$ along the ideal generated by $(t_1\.t_l,m_1\.m_r)$. The latter was computed in \S\ref{blowupsec}, and we saw that, indeed, its charts are of the form $\Spec(\ZZ[Q,\frac{t_i}y])$.
\end{proof}

\subsubsection{Functoriality}
In the sequel, by a {\em permissible blowing up} we mean the normalized blowing up along a permissible center. To simplify the notation, we will omit the normalization and will simply write $Bl_I(X)$. Naturally, permissible blowings up are compatible with logarithmically smooth morphisms.

\begin{lemma}\label{functorblowup}
Let $X$ be a logarithmic manifold and let $Z\into X$ be a permissible center. Then for any logarithmically smooth morphisms $f\:Y\to X$ of toroidal schemes, the pullback $T=Z\times_XY$ is a permissible center and $Bl_T(Y)=Bl_Z(X)\times_XY$ in the category of fs logarithmic schemes.
\end{lemma}
\begin{proof}
We know that $T$ is permissible by Lemma~\ref{centersfunct}. The problem is local on $X$ hence we can assume that $X$ is local. As in the proof of Lemma~\ref{permislem}, $Z=V(t_1\.t_l,m_1\.m_r)$ and $Z$ becomes monomial once we replace $U=X(0)$ by $U'=U\setminus\cup_{i=1}^l V(t_i)$. Since the pullbacks of $t_i$ form a subfamily of a regular family at any point of $f^{-1}(x)$, we also have that $V'=Y(0)\setminus \cup_{i=1}^lf^{-1}(V(t_i))$ defines a toroidal structure and $T$ is toroidal on $(Y,V')$. We omit the easy check that the morphism $(Y,V')\to(X,U')$ is logarithmically smooth. The lemma now follows from the fact that toroidal blowings up are compatible with logarithmically smooth morphisms.
\end{proof}

\subsection{Kummer ideals}
In \cite{ATW-principalization} we will also use a generalization of permissible blowings up that we are going to define now. Informally speaking, we will blow up ``ideals" of the form $(t_1\.t_n,m_1^{1/d}\.m_r^{1/d})$. Our next aim is to formalize such objects, and the main task is to define ``ideals" $(m^{1/d})$.

\subsubsection{Ideals $I^{[1/d]}$}
First, let us describe the best approximation to extracting roots on the logarithmic manifold itself. For any monomial ideal $I$ and $d\ge 1$ let $I^{[1/d]}$ denote the monomial ideal $J$ generated by monomials $m$ with $m^d\in I$. Recall that monomial ideals are in a one-to-one correspondence with the ideals of $\oM_X$. If $I$ corresponds to $J\subseteq\oM_X$ then $I^{[1/d]}$ corresponds to $\frac{1}{d}J\cap\oM_X$. So, extracting the root is a purely monomial operation, and hence it is compatible with strict morphisms $f\:Y\to X$ in the sense that $(f^{-1}(I))^{[1/d]}=f^{-1}\left(I_X^{[1/d]}\right)$.

\begin{remark}\label{rootidealrem}
It may happen that $I$ is invertible but $I^{[1/d]}$ is not. On the level of monoids this can be constructed as follows: take $M\subset\ZZ^2$ given by $x+y\in 3\ZZ$ and $I=(3,3)+M$. Then $I^{[1/3]}$ is generated by $(1,2)$ and $(2,1)$ and it is not principal.
\end{remark}

\subsubsection{Kummer monomials}
By a {\em Kummer monomial} on a logarithmic scheme $X$ we mean a formal expression $m^{1/d}$ where $m$ is a monomial on $X$ and $d\ge 1$ is an integer which is invertible on $X$. In order to view $m^{1/d}$ as an actual function we should work locally with respect to a certain log-\'etale topology. For example, $X[m^{1/d}]:=(X\otimes_{k[m]}k[m^{1/d}])^\sat$ is the universal fs logarithmic scheme over $X$ on which $m^{1/d}$ is defined, and $X[m^{1/d}]\to X$ is logarithmically \'etale by our assumption on $d$.

\begin{remark}
One can also consider roots with a non-invertible $d$ but then the morphism $X[m^{1/d}]\to X$ is only logarithmically syntomic, i.e. logarithmically flat and lci. We prefer to exclude such cases because we will later consider only toroidal schemes, and logarithmic regularity is not local with respect to the log-syntomic topology.
\end{remark}

\subsubsection{Kummer topology}\label{Sec:Kummer-topology}
In order to define operations on different monomials one has to pass to larger covers of $X$, and there are two ways to do this uniformly. The first one is to consider the pro-finite coverings and work with structure sheaves on non-noetherian schemes, see \cite{Talpo-Vistoli}. Another possibility is to work with the structure sheaf of a topology generated by finite coverings. The two approaches are equivalent. We adopt the second one using the Kummer logarithmically \'etale topology defined by Nizio{\l}  in \cite{Niziol-K-theory-of-log-schemes-I}. For brevity, it will be called the Kummer topology.

Recall that a logarithmically \'etale morphism $f\:Y\to X$ is called {\em Kummer} if for any point $y\in Y$ with $x=f(y)$ the homomorphism $\oM^\gp_x\to\oM^\gp_y$ is injective with finite cokernel, and $\oM_y$ is the saturation of $\oM_x$ in $\oM^\gp_y$. Setting surjective Kummer morphisms to be coverings, we obtain a {\em Kummer topology} on the category of fs logarithmic schemes. The site of Kummer logarithmic schemes over $X$ will be denoted $X\ket$. The following lemma shows that when working with the Kummer topology it suffices to consider two special types of coverings. The proof is simple, and we refer to \cite[Corollary~2.17]{Niziol-K-theory-of-log-schemes-I} for details.

\begin{lemma}\label{Kummercovering}
The topology of $X\ket$ is generated by two types of coverings: strict \'etale morphisms $Z\to Y$ and morphisms of the form $Y[m^{1/d}]\to Y$.
\end{lemma}

\subsubsection{The structure sheaf}
The rule $Y\mapsto\Gamma(\cO_Y)$ defines a presheaf of rings $\cO_{X\ket}$ on $X\ket$.

\begin{lemma}
The presheaf $\cO_{X\ket}$ is a sheaf.
\end{lemma}
\begin{proof}
A more general claim is proved in \cite[Proposition~2.18]{Niziol-K-theory-of-log-schemes-I}. Let us outline a simple argument that works in our case. It suffices to check the sheaf condition for the two coverings from Lemma~\ref{Kummercovering}. The first case is clear since $\cO_{X\et}$ is a sheaf. In the second case we note that $\mu_d$ acts on $Y'=Y[m^{1/d}]$ and $Y$ is the quotient, in particular, $\cO_Y(Y')^{\mu_d}=\cO_Y(Y)$. The saturated fiber product $Y''=(Y'\times_YY')^\sat$ equals to $\mu_d\times Y'$, hence the coequalizer of  $\cO_Y(Y'')\toto\cO_Y(Y')$ equals to $\cO_Y(Y')^{\mu_d}$, that is, $\cO_Y$ satisfies the sheaf condition with respect to the covering $Y'\to Y$.
\end{proof}

\subsubsection{Kummer ideals}
By a Kummer ideal we mean an ideal $I\subseteq\cO_{X\ket}$ which is coherent in the following sense: there exists a Kummer covering $Y\to X$ and a coherent ideal $I_Y\subseteq\cO_Y$ such that $I|_{Y\ket}$ is generated by $I_Y$ in the sense that $\Gamma(Z,I)=\Gamma(Z,I_Y\cO_Z)$ for any Kummer morphism $Z\to Y$.

\begin{example}\label{Qidealexam}
(i) If $I_X$ is a monomial ideal on $X$ let $I$ be the associated ideal on $X\ket$ and for $Y$ Kummer over $X$ let $I_Y$ denote restrictions of $I$ onto $Y$. Given  $d\ge 1$ define $J=I^{1/d}$ by $J_Y=(I_Y)^{[1/d]}$. Note that the projections $p_{1,2}$ of $Z=(Y\times_XY)^\sat$ onto $Y$ are strict. Hence $p_i^{-1}(J_Y)=J_Z$ for $i=1,2$, and we obtain that the pullbacks are naturally isomorphic, that is, $J$ is an ideal in $\cO_{X\ket}$. Moreover, $J$ is coherent because one can easily construct a covering $Y\to X$ such that $I_Y=J_Y^d$ and then $J_Z=J_Y\cO_Z$ for any Kummer morphism $Z\to Y$. For example, choose an open covering $X=\cup_i X_i$ such that the ideals $I|_{X_i}=(m_i)$ are principal and take $Y=\coprod_i X_i[m_i^{1/d}]$.

(ii) One can produce more ideals using addition and multiplication, ideals coming from $\cO_X$, and Kummer ideals from (i). For example, if $t_i\in\Gamma(\cO_X)$ and $m_j$ are global monomials then the ideal $J=(t_1\.t_n,m_1^{1/d}\.m_r^{1/d})$ is a well-defined coherent Kummer ideal, as well as its powers $J^l$.
\end{example}

\begin{remark}\label{Qidealrem}
(i) It is very essential that we are working with saturated logarithmic schemes and the Kummer topology. For example, if $X=\Spec(k[t])$ and $X\fl$ denotes the small flat site of $X$ then by the usual flat descent $\cO_{X\fl}$ is a sheaf and any its ideal comes from an ideal of $\cO_X$. In particular, the ideal $t\cO_{X\fl}$ is not a square. This happens for the following reason: although $(t)=(y^2)$ on the double covering $Y=\Spec(k[y])\to X$ with $y^2=t$, the fiber product $Z=Y\times_XY$ equals to $\Spec(k[y_1,y_2]/(y_1^2-y_2^2))$ and the two pullbacks of $(y)$ to $Z$ are different: $(y_1)\neq(y_2)$. In other words, the root $(y)=\sqrt{(t)}$ is not unique locally on $X\fl$ and hence does not give rise to an ideal.

(ii) The sheaf $\cO_{X\ket}$ also has non-coherent ideals. For example, for $X=\Spec(k[m])$ the maximal monomial ideal $\sum_{i=1}^\infty(m^{1/d})$ is not finitely generated.
\end{remark}

\subsection{Blowings up of permissible Kummer ideals}

\subsubsection{Permissible Kummer centers}\label{Sec:permissible-center}
We restrict our consideration to toroidal schemes. Permissible centers extend to Kummer ideals straightforwardly: we say that a Kummer ideal $I$ on a toroidal scheme $X$ is {\em permissible} if it is generated by the ideal of a toroidal subscheme and a monomial Kummer ideal. In other words, for any point $x\in X$ one has that $I_x=(t_1\.t_n,m_1^{1/d}\.m_r^{1/d})$, where $t_1\.t_n$ is a part of a regular sequence of parameters, $n$ is invertible on $X$, and $m_1\.m_r$ are monomials. By $V(I)$ we denote the set of points of $X$ where $I$ is not the unit ideal; it is a closed subset of $X$.

\subsubsection{Kummer blowings up: global quotient case}
Let $I$ be a permissible Kummer center on $X$. The idea of defining $Bl_I(X)$ is to blow up a sufficiently fine Kummer covering of $X$ and then descend it to a modification of $X$.

Assume first that there exists a $G$-Galois Kummer covering $Y\to X$ such that $I$ is generated by $I_Y$. Note that $X=Y/G$. Locally, $I_Y$ is generated by monomials and elements coming from $I$. Since $G$ acts by characters on monomials and preserves elements coming from $I$, the ideal $I_Y$ and the blowing up $Y'=Bl_{I_Y}(Y)\to Y$ are $G$-equivariant. Moreover, using these generators we see that the blowing up $Y'$ is covered by $G$-equivariant affine charts. In particular, the algebraic space $X'=Y'/G$ is a scheme, and $X'\to X$ is a $W$-modification, where $W=X\setminus V(I)$; here a \emph{$W$-modification} $X'\to X$ is a modification restricting to the identity over the dense open $W\subset X$.

Note that $X'$ is the coarse space $[Y'/G]_\cs$ of the stack quotient $[Y'/G]$. We will show that $X'$ depends only on $X$ and $I$, but it may happen that $X'$ with the quotient logarithmic structure is not toroidal: see \S\ref{Kummerchart} below for a general explanation and Example~\ref{orbifoldexam} for a concrete example. On the other hand, $[Y'/G]$ is too close to $Y'$: the morphism $Y'\to[Y'/G]$ is \'etale hence $[Y'/G]$ is toroidal, but it is ramified over the same points of $X'$ over which $Y'$ is ramified, and hence depends on the choice of the covering $Y\to X$. Finally, we would like to ensure that the exceptional divisor $E$ on $[Y'/G]$ remains Cartier, in other words, we would like the morphism $[Y'/G] \to B\GG_m$ corresponding to the line bundle $\cO(E)$ to descend to our modification. For these reasons the main player in the sequel will be the relative coarsening $[Y'/G]_{\cs/B\GG_m}$ (see \S\ref{Sec:coarsening} and Remark~\ref{coarserem}). In particular, we will see that it is toroidal and independent of the choice of the covering $Y\to X$.

\begin{lemma}\label{Qblowlem0}
With the above notation, the $X$-stack $\cX'=[Y'/G]_{\cs/B\GG_m}$ and its coarse space $X'=Y'/G$ depend on $X$ and $I$ only, but not on the Kummer covering $Y\to X$.
\end{lemma}
\begin{proof}
It suffices to deal with $\cX'$, since $X'$ is obtained from it. We should prove that if $Z\to X$ is another Kummer covering with Galois group $H$ and $Z'=Bl_{I_Z}(Z)$ then $[Z'/H]_{\cs/B\GG_m}=\cX'$. The family of Kummer coverings is filtered, hence it suffices to consider the case when $Z$ dominates $Y$. In this case, $Z/K=Y$ where $K$ is a subgroup of $H$ with $H/K=G$.

Since $I_Z=I_Y\cO_Z$, the charts of both $Bl_{I_Y}(Y)$ and $Bl_{I_Z}(Z)$ can be given by the same elements. It follows that $Z'\to Y$ factors through a finite morphism $Z'\to Y'$. Since $Y'$ is normal, this implies that $Z'/K=Y'$, and we obtain a coarsening morphism $h\:[Z'/H]\to[Y'/G]$. Clearly, the exceptional divisor on $[Z'/H]$ is the pullback of the exceptional divisor on $[Y'/G]$. Therefore the morphism $[Z'/H]\to B\GG_m$ factors through the morphism $[Y'/G]\to B\GG_m$, and this implies that $[Z'/H]_{\cs/B\GG_m}=[Y'/G]_{\cs/B\GG_m}$, as required.
\end{proof}

\subsubsection{Kummer blowings up: the general case}\label{Sec:Kummer-blowup}
In the general case, the Kummer blowing up of $X$ along $I$ are defined by gluing. Namely, $X$ is covered by open subschemes $X_i$ such that $I_i=I|_{X_i}$ is generated by global functions and roots of global monomials, and then each $X_i$ has a $G_i$-Kummer Galois covering $Y_i\to X_i$ such that $J_i=I_{Y_i}$ generates $I|_{Y_i}$. By Lemma~\ref{Qblowlem0} the stack $\cX'_i=[Bl_{J_i}(Y_i)/G_i]_{\cs/B\GG_m}$ and its coarse space $X'_i=Bl_{J_i}(Y_i)/G_i$ depend on $X_i$ and $I_{X_i}$ only. Moreover, the uniqueness of the factorizations $[Bl_{J_i}(Y_i)/G_i]\to\cX'_i\to X_i$ (see Theorem~\ref{coarseth}(ii)) implies that $\cX'_i$ glue over the intersections $X_i\cap X_j$. For example, since open immersions are inert morphisms we can use here Lemma~\ref{inertquot}. Thus, we obtain morphisms $\cX'\to X$ and $X'\to X$ depending only on $X$ and $I$. We say that $X':=Bl_I(X)$ is the {\em coarse Kummer blowing up} of $X$ along $I$ and $\cX'=[Bl_I(X)]$ is the {\em Kummer blowing up} of $X$ along $I$. Here are two basic properties of this operation.

\begin{theorem}\label{Qblowlem}
Assume that $(X,U)$ is a toroidal scheme and $I$ is a permissible center, and let $W=X\setminus V(I)$. Then

(i) $f\:[Bl_I(X)]\to X$ and $Bl_I(X)\to X$ are $W$-modifications of $X$,

(ii) $([Bl_I(X)],f^{-1}(U))$ is a simple toroidal orbifold.
\end{theorem}
\begin{proof}
The claims are local on $X$, so we can assume that $X$ possesses a $G$-Galois Kummer covering $Y$ such that $I_Y$ generates $I|_{Y\ket}$. Then $[Bl_{I_Y}(Y)/G]$ is proper over $X$ and the preimage of $W$ is dense, and hence the same is true for the partial coarse spaces $[Bl_I(X)]$ and $Bl_I(X)$. Furthermore, the constructions are compatible with localizations and $I|_W=1$, hence both are $W$-modifications of $X$.

The fact that $([Bl_I(X)],f^{-1}(U))$ is a toroidal orbifold is shown in Lemma~\ref{torlem} below, using the explicit charts described in Section \ref{Kummerchart}. Its simplicity follows from the observation that $G$ acts simply on $Y$, and hence it also acts simply on $Bl_{I_Y}(Y)$.
\end{proof}

\subsubsection{Charts of Kummer blowings up}\label{Kummerchart}
Next, let us describe explicit charts of Kummer blowings up. Assume that $X=\Spec(A)$ and $I=(t_1\.t_n,m_1^{1/d}\.m_r^{1/d})$ a permissible Kummer ideal, where $(t_1\.t_n)$ defines a toroidal subscheme and $m_i$ are global monomials. Then $X'=[Bl_I(X)]$ is of the form $[Bl_J(Y)/G]_{\cs/B\GG_m}$, where $$A'=A\otimes_{\ZZ[m_1\.m_r]}\ZZ[m_1^{1/d}\.m_r^{1/d}]),$$ $Y=\Spec(A')$, $G=(\ZZ/d\ZZ)^r$, and $J=I\cO_Y$. Note that $Bl_J(Y)$ is covered by the charts $$Y'_y=\Spec(A'[t'_1\.t'_n,u'_1\.u'_r]^\sat),$$ where $y\in\{t_1\.t_n,m_1^{1/d}\.m_r^{1/d}\}$, $t'_i=\frac{t_i}y$ and $u'_j=\frac{m_i^{1/d}}y$. Hence $X'$ is covered by the charts $X'_y=[Y'_y/G]_{\cs/B\GG_m}$.

Let us describe $X'_y$ locally at the image of a point $q\in Y'_y$. The stabilizer $G_q$ is the inertia group of $[Y'_y/G]$ at the image of $q$. Hence the morphism $[Y'_y/G]\to B\GG_m$ induces a homomorphism $G_q\to\GG_m$, whose kernel $G_{q/B\GG_m}$ is the relative stabilizer of $[Y'_y/G]$ over $\GG_m$ at the image of $q$. In particular, $X'_y=[(Y'_y/G_{q/B\GG_m})/(G/G_{q/B\GG_m})]$ locally at the image of $q$. To complete the picture it remains to observe that the relative stabilizer $G_{q/B\GG_m}$ is the subgroup of $G_q$ acting trivially on $y$, that is, $G_q$ acts on $y$ through its image in $\GG_m$. To spell this explicitly consider two cases:

(1) The $t_i$-chart. Since $G$ acts trivially on $t_i$ we have that $G_{q/B\GG_m}=G_q$ and hence $X'_y=Y'_y/G$ is a scheme.


(2) The $m_i^{1/d}$-chart. In this case, $G_{q/B\GG_m}$ contains $(\ZZ/d\ZZ)^{r-1}$ and $G/G_{q/B\GG_m}=\ZZ/e\ZZ$, where $e$ is the minimal divisor of $d$ such that $m_i\in\cM_x^{d/e}$, where $x\in X$ is the image of $q$; in particular, $G_q$ acts through $\ZZ/e\ZZ$ on the image of $m_i^{1/d}$ in $\cM_q$.


\begin{lemma}\label{torlem}
Keep the above notation. Then the group $G_{q/B\GG_m}$ acts toroidally at $q$. In particular, the coarsening $[Y'/G]\to[Bl_I(X)]$ is toroidal and $[Bl_I(X)]=[Y'/G]_{\cs/\GG_m}=[Y'/G]_{\tcs/B\GG_m}$.
\end{lemma}
\begin{proof}
The regular coordinates on $Y'_y$ are of the form $t'_i=\frac{t_i}{y}$. Since $G_{q/B\GG_m}$ acts trivially on $t_i$ and $y$, it acts trivially on $t'_i$. Thus, its action at $q$ is toroidal.
\end{proof}

We will not need the following remark, so its justification is left to the interested reader.

\begin{remark}
(i) The whole group $G_q$ can act non-trivially on $m_i^{1/d}$-charts, see Example~\ref{orbifoldexam}(ii) below. So, one may wonder what is the maximal toroidal coarsening $[Y'/G]_\tcs$. By the above lemma, we have a natural morphism $f\:[Bl_I(X)]\to[Y'/G]_\tcs$. It turns out that in the non-monomial case (i.e., there exists at least one regular parameter $t_1$), $f$ is an isomorphism. On the other hand, in the monomial case the action of the whole $G_q$ is automatically toroidal, and hence $[Y'/G]_\tcs=Y'/G$. In this case, $f$ can be a non-trivial coarsening, see Example~\ref{orbifoldexam}(i).

(ii) In a first version of the paper, we defined $[Bl_I(X)]$ to be equal to $[Y'/G]_\tcs$. This definition possesses worse functorial properties and often required to distinguish the monomial and non-monomial cases. It seems that the new definition is the ``right'' one.
\end{remark}

\subsubsection{The coarse blowing up}
The coarse blowing up can be computed directly.

\begin{lemma}\label{coarseblow}
Assume given a toroidal affine scheme $X=\Spec(A)$ with a Kummer ideal $I=(t_1\.t_n,m_1^{1/d}\.m_r^{1/d})$ and a positive number $e\in d\ZZ$. Then $Bl_I(X)$ is the normalized blowing up of $X$ along either of the following ideals: $J_e=(t_1^e\.t_n^e,m^{e/d}_1\.m^{e/d}_r)$, $\tilJ_e=I^e\cap\cO_X$.
\end{lemma}
\begin{proof}
Set $Y=Spec(B)$ with $B=A[m_1^{1/d}\.m_r^{1/d}]$. It suffices to check that $Bl_{I_Y}(Y)$ is finite over both $Bl_{J_e}(X)$ and $Bl_{\tilJ_e}(X)$. Indeed, in this case $Bl_I(X)=Bl_{I_Y}(Y)/(\ZZ/d\ZZ)^r$ is a finite modification of both $Bl_{J_e}(X)^\nor$ and $Bl_{\tilJ_e}(X)^\nor$, and since the latter are normal we are done.

We will check the finiteness on charts. Let $y\in\{t_1\.t_n,m^{1/d}_1\.m^{1/d}_r\}$ and $x=y^e$. It suffices to show that $B[I/y]$ is finite over both $A[J_e/x]$ and $A[\tilJ_e/x]$. But this is clear because $B[I/z]$ is integral over both $B[J_eB/z]$ and $B[\tilJ_eB/z]$.
\end{proof}

\subsubsection{Examples}
Let us consider two basic examples of Kummer blowings up.


\begin{example}\label{orbifoldexam}
(i) Let $X=\Spec(k[\pi])$ with the logarithmic structure given by $\pi$, and let $I=(\pi^{1/d})$. Then $Bl_I(X)=[\Spec(k[\pi^{1/d}])/\mu_d]$ has stabilizer $\mu_d$ at the origin.

(ii) Let $X=\Spec(k[t,\pi])$ with the logarithmic structure given by $\pi$, and let $I=(t,\pi^{1/2})$. By Lemma~\ref{coarseblow}, the coarse blow up $X'=Bl_I(X)$ coincides with $Bl_J(X)^\nor$, where $J=(t^2,\pi)$. In fact, $Bl_J(X)$ is already normal and covered by two charts: $X'_1=\Spec(k[t,\pi,\frac{t^2}\pi])$ and $X'_2=\Spec(k[t,\frac{\pi}{t^2}])$. The chart $X'_2$ is regular, but the chart $X'_1$ has an orbifold singularity $O_X$ at the origin. Moreover, the natural logarithmic structure on $X'_1$ is generated by $\pi$ only, and $X'_1$ is not toroidal with this logarithmic structure. (Though $X'_1$ can be made toroidal by increasing the toroidal structure, for example, by adding the divisor $(t)$.)

Now let us consider the finer stack-theoretic picture. The Kummer blowing up $\cX'=[Bl_I(X)]$ can be computed using the Kummer covering $Y=\Spec(k[t,\pi^{1/2}])$ with $G=\ZZ/2\ZZ$. This can be done directly, but for the sake of comparison we will first compute $\cX''=[Y'/G]_\tcs$, where $Y'=Bl_{(t,\pi^{1/2})}(Y)$. Cover $Y'$ by two charts: $Y'_1=\Spec(k[\frac{t}{\pi^{1/2}},\pi^{1/2}])$ and $Y'_2=\Spec(k[t,\frac{\pi^{1/2}}{t}])$, then $\cX''$ is covered by the charts $\cX''_i=[Y'_i/G]_\tcs$. The action of $G$ on $Y'_2$ is toroidal, and hence $\cX''_2=Y'_2/G=X'_2$. The action of $G$ at the origin $O_Y$ of $Y'_1$ is not toroidal because $G$ acts via the non-trivial character on both parameters. Therefore the stabilizer at the image $O_\cX\in\cX''$ of $O_Y$ is $G$. In particular, the coarse moduli space $\cX''\to X'$ is an isomorphism over $X'\setminus\{O_X\}$, and the preimage of $O_X$ is the point $O_\cX$ with a non-trivial stack structure. Furthermore, it is easy to see that the exceptional divisor is Cartier on $\cX''$, and hence the morphism $\cX'\to\cX''$ admits a section. Thus, $\cX'=\cX''$ is the cone orbifold.
\end{example}

\subsubsection{Enlarging the toroidal structure}
As in the proof of Lemma~\ref{permislem}, enlarging the torodial structure any Kummer blowing up can be made a toroidal morphism.

\begin{lemma}\label{enlargetoroidal}
Let $X=(X,U)$ be a toroidal scheme, $I$ a permissible Kummer ideal on $X$ and $f\:X'=[Bl_I(X)]\to X$ the associated Kummer blowing up. Assume that $X_1=(X,U_1)$ is a toroidal scheme obtained by enlarging the toroidal structure so that $I$ is monomial on $X_1$ (see Corollary~\ref{increasecor}). Then $X'_1=(X',f^{-1}(U_1))$ is a toroidal orbifold and the morphism $X'_1\to X_1$ is toroidal.
\end{lemma}
\begin{proof}
The claim is local on $X$, hence we can assume that there exists a $G$-Galois Kummer covering $Y\to X$ such that $J=I\cO_Y$ is a permissible ideal. Let $Y'=Bl_J(Y)$ and let $Y'_1$ and $Y_1$ be the toroidal schemes with the toroidal structure induced from $U_1$. Since $J$ is monomial on $Y_1$, we have that $Y'_1\to Y_1$ is a toroidal blowing up. By \S\ref{Kummerchart} the action of $G$ on $Y'_1$ is toroidal (it acts trivially on all regular coordinates). Therefore, any subgroup $H\subseteq G$ acts toroidally and hence the morphism $Y'_1/H\to Y_1$ is toroidal. It follows that for any coarsening $T$ of $[Y'_1/G]$ the morphism $T\to Y_1/G=X_1$ is toroidal. It remains to recall that, by definition, $X'$ is a coarsening of $[Y'/G]$, namely the relative coarse space with respect to the morphism $[Y'/G]\to B\GG_m$ induced by the exceptional divisor.
\end{proof}

\subsubsection{The universal property}
Kummer blowings up can be characterized by a universal property which extends the classical characterization of blowings up.
\begin{theorem}\label{Th:permissible-Kummer}
Let $X$ be a toroidal scheme and let $I$ be a permissible Kummer ideal with the associated Kummer blowing up $f\:[Bl_I(X)]\to X$. Then $f^{-1}(I)$ is an invertible ideal and $f$ is the universal morphism of toroidal orbifolds $h\:Z\to X$ such that $h^{-1}(I)$ is an invertible ideal.
\end{theorem}

\begin{proof}
All claims are local on $X$, so we can use the description of charts from \S\ref{Kummerchart}: choosing a $G$-Galois Kummer covering $Y\to X$, such that $I_Y$ is an ordinary ideal, and setting $Y'=Bl_{I_Y}(Y)$ we have that $[Bl_I(X)]=[Y'/G]_{\cs/B\GG_m}$. Now, the first claim is obtained by unravelling the definition of $X':=[Bl_I(X)]$. Indeed, the exceptional divisor on $Y'$, and hence also on $Y'/G$, is Cartier. Furthermore, the induced morphism $[Y'/G]\to B\GG_m$ factors through $X'$, that is the exceptional divisor on $X'$ is also Cartier.

Now, let us check the universal property. So, assume that $h\:Z\to X$ is such that $h^{-1}(I)$ is an invertible ideal, and let us show that it factors through $[Bl_I(X)]$ uniquely up to a unique 2-isomorphism. Set $T=Z\times_XY$ as an fs logarithmic scheme. From the factorization $T\to Z\to X$, the pullback of $I$ to $T$ is an invertible Kummer ideal. From the factorization $T\to Y\to X$, the pullback of $I$ to $T$ is the usual ideal $I_Y\cO_T$. Therefore $I_Y\cO_T$ is an invertible ideal, and by the universal property of blowings up, $T\to Y$ factors through a morphism $T\stackrel \phi\to Y'=Bl_{I_Y}(Y)$ in a unique way. The exceptional divisors on $T$ and $Y'$ are compatible, hence induce compatible morphisms to $B\GG_m$.

Note that $T\to Z$ is Kummer \'etale with Galois group $G=(\ZZ/d\ZZ)^r$ equal to the Galois group of $Y\to X$. Taking the stack quotient by $G$, the exceptional divisors remain Cartier, hence morphisms $[T/G]\to[Y'/G]\to B\GG_m$ arise. Passing to the relative coarse moduli spaces yields a morphism $[T/G]_{\cs/B\GG_m}\to X'$. It remains to recall that the exceptional divisor on $Z=T/G$ is already Cartier, hence $[T/G]_{\cs/B\GG_m}=Z$ and we obtain the required morphism $Z\to X'$.
\end{proof}

\subsubsection{Strict transforms}
By a classical observation, the universal property of blowings up implies that if $X'\to X$ if the blowing up along an ideal $I$ then the strict transform $Z'$ of a closed subscheme $Z\into X$ is the blowing up of $Z$ along $I\cO_Z$. The same reasoning applies to Kummer blowings up as well.

\begin{lemma}\label{strlem}
Assume that $X$ is a toroidal scheme, $Z\into X$ is a closed toroidal subscheme, and $I\subseteq\cO_X$ is a permissible Kummer ideal whose restriction $J=I\cO_Z$ is a permissible Kummer ideal on $Z$. Let $X'\to X$ be the Kummer blowing up along $I$ and let $Z'$ be the strict transform of $Z$ (i.e., the closure of $Z\setminus V(I)$ in $X'$). Then the morphism $Z'\to Z$ factors through a unique isomorphism $Z'=[Bl_J(Z)]$.
\end{lemma}
\begin{proof}
On the one hand, since $Z'\to X$ factors through $X'$, the ideal $I\cO_{Z'}=J\cO_{Z'}$ is invertible. So, $Z'\to Z$ factors through a morphism $h\:Z'\to Y=[Bl_J(Z)]$ by Theorem~\ref{Th:permissible-Kummer}. On the other hand, $J\cO_Y$ is an invertible ideal, and since $J\cO_Y=I\cO_Y$, we obtain by Theorem~\ref{Th:permissible-Kummer} that the morphism $Y\to X$ factors through $X'$. Furthermore, $Y\to X$ factors through $Z'$ because $Z\setminus V(J)$ is dense in $Y$. This provides a morphism $Y\to Z'$, which is easily seen to be the inverse of $h$ by the uniqueness of the factorization in Theorem~\ref{Th:permissible-Kummer}.
\end{proof}

Since Kummer blowings up were only defined for toroidal orbifolds, we cannot extend the above theorem to the case when $Z$ is an arbitrary closed logarithmic substack of $X$. However, in this case we can at least describe the strict transform on the level of the coarse space.

\begin{lemma}\label{strlem2}
Assume that $X$ is a toroidal scheme, $Z\into X$ is a strict closed logarithmic subscheme, and $I\subseteq\cO_X$ is a permissible Kummer ideal. Let $X'\to X$ be the Kummer blowing up along $I$ and let $Z'\to Z$ be the strict transform. Set $J_n=I^{n!}\cap\cO_X$. Then $Z'_\cs$ is the blowing up of $Z$ along $((J_n)^m)^\nor\cO_Z$ for large enough $n$ and $m$.
\end{lemma}
\begin{proof}
The claim is local on $X$, hence by Lemma~\ref{enlargetoroidal} we can enlarge the logarithmic structure on $X$ making $I$ monomial. Recall that by Lemma~\ref{coarseblow}, $X'_\cs\to X$ is the normalized blowing up along $J_n$ for a large enough $n$. Clearly $J_n$ is monomial, hence by \cite[Corollary~5.3.6]{AT1} $X'_\cs\to X$ is the blowing up along $((J_n)^m)^\nor$ for a large enough $m$. Note that $Z'_\cs$ is the closed subscheme of $X'_\cs$ coinciding with the image of $Z'$. It follows that $Z'_\cs$ is the strict transform of $Z$ and hence it is the blowing along $((J_n)^m)^\nor\cO_Z$ by the usual theory of strict transforms.
\end{proof}

\subsubsection{Functoriality}
The universal property can also be used to show that, as most other constructions of this paper, Kummer blowings up are compatible with logarithmically smooth morphisms.

\begin{lemma}\label{kummerfunctor}
Let $f\:Y\to X$ be a logarithmically smooth morphisms of toroidal schemes, $I$ a permissible Kummer center on $X$, and $J=f^{-1}(I)$. Then $[Bl_J(Y)]=[Bl_I(X)]\times_XY$, where the product is taken in the category of fs logarithmic schemes.
\end{lemma}
\begin{proof}
Recall that $J$ is permissible by Lemma~\ref{functorblowup}. Set $X'=[Bl_I(X)]$ and $Y'=[Bl_J(Y)]$. Since $J\cO_{Y'}=I\cO_{Y'}$, the morphism $Y'\to X$ factors through $X'$ by Theorem~\ref{Th:permissible-Kummer}, and we obtain a morphism $Y'\to X'\times_XY$. Conversely, since $X'\times_XY$ is logarithmically smooth over $X'$, the pullback of the invertible ideal $I\cO_{X'}$ to $X'\times_XY$ is also invertible. The latter coincides with the pullback of $J$ to $X'\times_XY$, and using  Theorem~\ref{Th:permissible-Kummer} again we obtain a morphism $X'\times_XY\to Y'$. It follows from the uniqueness of the factorizations that these two morphisms are inverse, implying the lemma.
\end{proof}

\subsection{Kummer blowings up of toroidal orbifolds}

\subsubsection{Kummer ideals}
The Kummer topology naturally extends to logarithmic stacks, giving rise to the notion of Kummer ideals. Permissibility of Kummer ideals is an \'etale-local notion hence it extends to toroidal orbifolds too. Also, Lemma~\ref{permislem}, which concerns usual coherent ideals, generalizes as follows:
\begin{quote} A permissible blowing up of a toroidal orbifold is again a toroidal orbifold. \end{quote}
To combine the two notions and form the \emph{Kummer} blowing up of a toroidal orbifold we must proceed with caution:
a priori a non-representable operation, such as a Kummer blowing up, does not have to descend (or can lead to a 2-stack as an output).

\subsubsection{Kummer blowings up}
Assume now that $X$ is a toroidal orbifold and $I$ is a permissible Kummer ideal on $X\ket$. Find a strict \'etale covering of $X$ by a toroidal scheme $X_0$ and set $X_1=X_0\times_XX_0$. The pullback $I_i$ of $I$ to $X_i$ is a permissible Kummer ideal, and we set $Y_i=[Bl_{I_i}(X_i)]$. Since $[X_1\toto X_0]$ is an \'etale groupoid whose projections and the multiplication morphism are strict, we obtain by Lemma~\ref{kummerfunctor} that $Y_1\toto Y_0$ is an \'etale groupoid \emph{of stacks} whose projections are strict and \emph{inert}. By Lemma~\ref{inertquot} the quotient $Y=[Y_0/Y_1]$ exists as a toroidal orbifold and satisfies $Y_i=X_i\times_XY$. We call $Y$ the Kummer blowing up of $X$ along $I$ and denote it $[Bl_I(X)]:=Y$. A straightforward verification  using Lemma~\ref{kummerfunctor} shows:
\begin{enumerate}
\item[(i)] The $X$-stack $Y=[Bl_I(X)]$ is independent of the presentation $X=[X_0/X_1]$ and depends only on $X$ and $I$. The uniqueness of $Y$ is understood up to an isomorphism of $X$-stacks, which is unique up to a unique 2-isomorphism. If $X$ is simple then $Y$ is simple.

\item[(ii)] If $f\:X'\to X$ is a logarithmically smooth representable morphism and $I'=f^{-1}(I)$ then $[Bl_{I'}(X')]=[Bl_I(X)]\times_XX'$, the product taken in the fs category.
\end{enumerate}

\appendix

\section{Torification}\label{torapp}

\subsection{The torification functors}

\subsubsection{The general case}
Let $W$ be a toroidal scheme acted on by a diagonalizable group $G$ in a relatively affine way. For example, any action of $G$ on a quasi-affine scheme is relatively affine. The main results of \cite{AT1} establish a so-called torification $\tilcT_{W,G}\:W'\longto W$, which is a $G$-equivariant sequence of blowings up of toroidal schemes such that the action on $W'$ is toroidal, see \cite[Theorems~4.6.5 and 5.4.5]{AT1}.
In addition, it is shown that the torification is compatible with strict strongly $G$-equivariant morphisms $f\:W'\to W$ in the sense that $\tilcT_{W',G'}$ is the {\em contracted pullback} of $\tilcT_{W,G}$, i.e. $f^*(\tilcT_{W,G})$ with all empty blowings up removed.

\subsubsection{Simple actions}
If the action is simple then slightly stronger results are available, see \cite[Theorems~4.6.3 and 5.4.2]{AT1}. In particular, torification is achieved by a single $G$-equivariant blowing up $\cT_{W,G}\:W'\longto W$, and the quotient morphism $\cT^0_{W,G}\:W'\sslash G\to W\sslash G$ has a natural structure of a blowing up.

\subsection{Stronger functoriality}
Using the methods of \cite{AT1} one can easily show that the functors $\tilcT$ and $\cT$ possess stronger functoriality properties than asserted there. Let us discuss this strengthening.

\subsubsection{$\lambda$-equivariance}\label{functorrem}
We start with an aspect that holds for both algorithms. Recall that a $G$-morphism $f\:W'\to W$ is {\em strongly equivariant} if $f$ is the base change of the GIT quotient $f\sslash G$. Some criteria of strong equivariance and related properties can be found in \cite[Theorem~1.3.1 and Lemma~5.6.2]{ATLuna}. In particular, if $f$ is strongly $G$-equivariant then it is strongly $H$-equivariant for any subgroup $H\subseteq G$. More generally, assume that $G'$ acts on $W'$, $G$ acts on $W$, and $f$ is $\lambda$-equivariant for a homomorphism $\lambda\:G'\to G$. We say that $f$ is {\em strongly $\lambda$-equivariant} if it is fix-point reflecting,
i.e. induces an isomorphism $G'_x=G_{f(x)}$ for any $x\in W'$, and strongly $H$-equivariant for any subgroup $H\subseteq G'$ such that $H\cap\Ker(\lambda)=1$.

\begin{theorem}\label{lambdath}
Assume that toroidal schemes $W$ and $W'$ are provided with relatively affine actions of diagonalizable groups $G$ and $G'$, respectively. Further assume that $\lambda\:G'\to G$ is a homomorphism, and a morphism $f\:W'\to W$ is surjective, strict and strongly $\lambda$-equivariant. Then $\tilcT_{W',G'}$ is the contracted pullback of $\tilcT_{W,G}$.
\end{theorem}
\begin{proof}
This happens because $\tilcT$ is defined in terms of local combinatorial data $(\oM_x,G_x,\sigma_x)$, see \cite[Section~3.6.8]{AT1}, and the latter only depends on $G_x$ rather than on the entire $G$.
\end{proof}

\subsubsection{Weakening the strictness assumption}
A finer observation is that the strictness assumption is not so essential for the functoriality of $\cT$. For comparison, note that $\tilcT$ is constructed using barycentric subdivisions which depend on the monoids $\oM_x$, hence it is not functorial with respect to non-strict morphisms.

\begin{theorem}\label{nostrictth}
Assume that toroidal schemes $W$ and $W'$ are provided with relatively affine and simple actions of diagonalizable groups $G$ and $G'$, respectively, $\lambda\:G'\to G$ is a homomorphism, and $f\:W'\to W$ is a surjective strongly $\lambda$-equivariant morphism. Further assume that for any point $x'\in W'$ with $x=f(x')$ the restriction $f_S\:S'\to S$ of $f$ to the logarithmic strata through $x'$ and $x$ is strongly $\lambda$-equivariant. Then $\cT_{W',G'}$ and $\cT^0_{W',G'}$ are the pullbacks of $\cT_{W,G}$ and $\cT^0_{W,G}$, respectively.
\end{theorem}
\begin{proof}
Note that a reference to \cite[Lemma~4.2.13(ii)]{AT1} is the only place in the proof of \cite[Theorems~4.6.3]{AT1}, where one uses the assumption that $f$ is strict. The lemma asserts that $f$ respects the reduced signatures: $f^*(\sigma_x)=\sigma_{x'}$. Recall that the latter are defined as the multisets of non-trivial characters through which $G_x$ acts on the cotangent spaces to $S$ and $S'$ at $x$ and $x'$, respectively. But we assume that $f_S$ is strongly $G_x$-equivariant, hence $f^*(\sigma_x)=\sigma_{x'}$ by \cite[Lemma~3.6.4]{AT1}, and we avoid the use of \cite[Lemma~4.2.13(ii)]{AT1}.
\end{proof}

\subsubsection{Logarithmically smooth morphisms}
Here is the main particular case of the above theorem that we will need.

\begin{corollary}\label{nostrictcor}
Assume that toroidal schemes $W$ and $W'$ are provided with relatively affine and simple actions of finite diagonalizable groups $G$ and $G'$, respectively, $\lambda\:G'\to G$ is a homomorphism, and $f\:W'\to W$ is a surjective, logarithmically smooth, fix-point reflecting, $\lambda$-equivariant morphism. Then $\cT_{W',G'}$ and $\cT^0_{W',G'}$ are the pullbacks of $\cT_{W,G}$ and $\cT^0_{W,G}$, respectively.
\end{corollary}
\begin{proof}
It suffices to prove that the induced morphisms $f_S\:S'\to S$ between the logarithmic strata are strongly equivariant. Since $f_S$ is logarithmically smooth, $f_S$ is smooth. Clearly, $f_S$ is fix-point reflecting. Since the groups are finite, all orbits are special and hence $f_S$ is inert (\cite[\S5.1.8 and \S5.5.3]{ATLuna}). Thus, $f_S$ is strongly equivariant (even strongly smooth) by \cite[Theorem~1.1.3(ii)]{ATLuna}.
\end{proof}

\bibliographystyle{amsalpha}
\bibliography{principalization}

\end{document}